\documentclass{amsart}

\usepackage{amsmath,amssymb, stmaryrd}
\usepackage{ascmac}
\usepackage{amsthm}
\usepackage{color} 
\usepackage{MnSymbol}

\theoremstyle{plain}
\newtheorem{thm}{Theorem}[section]

\newtheorem{prop}[thm]{Proposition}
\newtheorem{lmm}[thm]{Lemma}

\theoremstyle{definition}

\DeclareMathOperator{\an}{an}
\DeclareMathOperator{\hyb}{hyb}

\DeclareMathOperator{\NA}{NA}
\DeclareMathOperator{\MA}{MA}

\DeclareMathOperator{\Spec}{Spec}

\DeclareMathOperator{\Supp}{Supp}
\DeclareMathOperator{\Proj}{Proj}
\DeclareMathOperator{\Max}{Max}
\DeclareMathOperator{\ord}{ord}

\DeclareMathOperator{\NS}{NS}

\DeclareMathOperator{\dd^c}{dd^c}

\newcommand{\bA}{\mathbb{A}}
\newcommand{\bC}{\mathbb{C}}
\newcommand{\bD}{\mathbb{D}}

\newcommand{\bN}{\mathbb{N}}
\newcommand{\bP}{\mathbb{P}}
\newcommand{\bQ}{\mathbb{Q}}
\newcommand{\bR}{\mathbb{R}}
\newcommand{\bZ}{\mathbb{Z}}

\newcommand{\cC}{\mathcal{C}}
\newcommand{\cD}{\mathcal{D}}

\newcommand{\cF}{\mathcal{F}}
\newcommand{\cL}{\mathcal{L}}

\newcommand{\cO}{\mathcal{O}}

\newcommand{\cX}{\mathcal{X}}

\newcommand{\fm}{\mathfrak{m}}
\newcommand{\fa}{\mathfrak{a}}

\newcommand{\fA}{\mathfrak{A}}

\usepackage[truedimen,top=25truemm,bottom=25truemm,left=25truemm,right=25truemm]{geometry}

\title{Non-archimedean and hybrid dynamics of regular polynomial automorphisms}
\author{Irokawa Reimi}
\address{NTT Institute for Fundamental Mathematics, NTT Communication Science Laboratories, NTT Inc., 3-9-11 Midori-cho, Musashino-shi, Tokyo 180-8585, Japan}
\email{reimi.irokawa@ntt.com}

\date{2026/9/30}
\begin{document}

\begin{abstract}
    In this paper, we study the weak limit of the invariant measures attached to an analytic family $\{f_t\}_{t\in\bD^*}$ of regular polynomial automorphisms over $\bC^N$ parametrized by the unit punctured disk which may degenerate at the origin, by using the techniques called hybrid spaces developed by Boucksom--Favre--Jonsson. For each $t$, the invariant measure $\mu_t$ is constructed by Sibony, and we show that the weak limit of $\{\mu_t\}_t$ over the hybrid space is given by the invariant measure of the regular polynomial automorphism $f_{\bC((t))}^{\an}$ over $\bC((t))$ induced from the original family. This is a generalization of the author's previous result for degenerating families of dynamics of H\'enon mappings.
\end{abstract}
\maketitle

\section{Introduction}

In this paper, we study the degeneration of a meromorphic family of complex dynamics of regular polynomial automorphisms over $\bC^N$. 

We consider an analytic family $\{f_t:\bC^N\to\bC^N\}_{t\in\bD^*}$ of regular polynomial automorphisms of degree $d_+$ parametrized over the unit punctured disk. Here, we say a polynomial automorphism is regular if the indeterminacy loci of its extension and the extension of its inverse have empty intersection. We note that regular polynomial automorphisms are higher-dimensional generalizations of H\'enon maps in dimension 2, and
the dynamics of such maps are studied mainly by Sibony, especially their pluri-potential theoretic properties in several papers such as \cite{Sibony99}. Especially, he constructed invariant currents $T_f^{\pm}$ over $\bC^N$; these currents are closed and of bidegree $(1,1)$, with the equality
\begin{align*}
    (f^{\pm1})^*T_f^{\pm}=\frac{1}{d_{\pm}}T_f^{\pm}.
\end{align*}
Using these currents we may define the invariant measure $\mu_f$. In dimension 2, i.e.\ the dynamics of H\'enon maps, such a measure is already constructed by Berford and Smillie \cite{BedfordSmillie1991}.

In this paper, we are interested in the case where $\{f_t\}_{t\in\bD^*}$ meromorphically degenerates at the origin, that is, the extension $f_t:\bC^N\times \bD\to\bC^N\times \bD$ is a meromorphic map with indeterminacy only on the fiber of the origin. 

The theory of hybrid spaces is a good tool to study degenerating families of dynamical systems. The theory of hybrid spaces is developed/formulated in a series of papers by Boucksom, Favre and Jonsson \cite{Sebastien:2015uz} \cite{Sebastien:2017le}, \cite{Sebastien:2017lq} originally to study the degeneration of complex manifolds, and later applied to the study of dynamics by Favre \cite{Favre:2020rz}. We briefly recall the notion of hybrid spaces.

In our situation, each coefficient of $f_t$ is written as a meromorphic function defined on $\bD$ with poles only at the origin, which naturally induces a regular polynomial automorphism $f_{\bC((t))}:\bA^N_{\bC((t))}\to \bA^N_{\bC((t))}$ by taking Laurent expansion of the coefficients and regarding them as elements of $\bC((t))$. Introducing a $t$-adic norm on $\bC((t))$ and taking the Berkovich analytification, we may regard $f_{\bC((t))}^{\an}$, which we will denote by $f_{\NS}$ from now on, as a dynamical system on the non-archimedean variety. We note that we also study the non-archimedean dynamics of regular polynomial automorphisms in this paper; for details, see Section \ref{sec: regularpolyautom - Invariant measure of the non-archimedean dynamics}. Here we show results similar to Sibony \cite{Sibony99}, which will be used in the proof of our main result. Especially we construct the invariant measure $\mu_{\NA}$ attached to $f_{\NA}$.

The hybrid space gives a relation between the original family of complex dynamics $\{f_t\}$ and the induced non-archimedean dynamics $f_{\bC((t))}^{\an}$. We fix $r<1$ sufficiently small. The hybrid affine $N$-space $\bA^{N,\hyb}$ has a natural surjection $\pi^{\hyb}:\bA^{N,\hyb}\to\cC_{\hyb}(r)$ with a homeomorphism $\tau:\bar{\bD}_r\to \cC_{\hyb}(r)$ such that
\begin{align*}
    &\psi_t: \bC^N\to(\pi^{\hyb})^{-1}(\tau(t)) \text{ for }t\neq 0; \\
    &\psi_0: \bA^{N,\an}_{\bC((t))}\to (\pi^{\hyb})^{-1}(\tau(0))
\end{align*}
are both homeomorphisms. The map $f_t$ also induces a continuous map $f^{\hyb}$ such that, via the above homeomorphism $\psi_t$,
\begin{align*}
    &f^{\hyb}|_{t}=f_t \text{ for }t\neq 0; \\
    &f^{\hyb}|_0=f_{\bC((t))}^{\an}.
\end{align*}

Under this setting, our main result is as follows:

\begin{thm}
    Let $\{f_t:\bC^N\to\bC^N\}_{t\in\bD^*}$ be an analytic family of regular polynomial automorphisms. We set $\mu_t$ to be the measure of maximal entropy associated to $f_t$, $f_{\NA}:\bA^{N,\an}_{\bC((t))}\to\bA^{N,\an}_{\bC((t))}$ be the analytification of the automorphism naturally induced from the family $\{f_t\}$, and $\mu_{\NA}$ be the canonical measure attached to $f_{\NA}$. Then, over the associated hybrid space $\pi_{\hyb}:\bA^{N,\hyb}\to \cC_{\hyb}(r)$ and canonical homeomorphisms $\psi_t:\bC^N\to(\pi_{\hyb})^{-1}(\tau(t))$ for $0<|t|<r$ and $\psi_0:\bA^{N,\an}_{\bC((t))}\to(\pi_{\hyb})^{-1}(\tau(0))$, we have the following weak convergence of measures:
    \begin{align*}
        \lim_{t\to0}(\psi_t)_*\mu_t=(\psi_0)_*\mu_{\NA}.
    \end{align*}
\end{thm}

In the later sections, we will explain the definition of everything appearing in this theorem.

As written above, the theory of hybrid spaces was first applied to the study of degenerating complex dynamics by Favre \cite{Favre:2020rz}. He proved a similar weak convergence result for canonical measures for a degenerating family of holomorphic endomorphisms of projective spaces $\bP^n(\bC)$. The author \cite{irokawa2023hybrid} proved a similar result for degenerating families of complex H{\'e}non mappings, and our main result here is a generalization of this. As an application of Favre's result, Bianchi and Okuyama \cite{Bianchi2018DegenerationOQ} studied the degeneration of quadratic polynomial endomorphisms of $\bP^2(\bC)$ when these maps degenerate to an H{\'e}non map. We also note that, before Favre's work, DeMarco and Faber \cite{DE-MARCO:2014tn} studied the degeneration of rational functions on $\bP^1(\bC)$ using a similar technique: they used a Berkovich projective line over the completion of the Puiseux series field $\bC\{\{t\}\}$, and Okuyama \cite{okuyama2024degenerating} corrected their proof.

Here we clarify the difference between our result and the author's previous result in \cite{irokawa2023hybrid}, the similar weak convergence of the invariant measures for dimension $2$. The proof of the weak convergence, which is displayed in Section \ref{sec: regularpolyautom - Hybrid convergence of the invariant measures}, is a higher-dimensional analogue of that in \cite{irokawa2023hybrid}, though the invariance of the limit measure was not known at that time; it will be proved in Section \ref{sec: regularpolyautom - Invariant measure of the non-archimedean dynamics}. That is, the claim that the measure $\mu_{\NA}$ is invariant under the action of $f$, is new even in the dimension $2$ case.

\subsection*{Acknowledgement}
I would like to thank Professor Charles Favre for suggesting this problem as a direction for future work upon receiving \cite{irokawa2023hybrid}, and Professor Gabriel Vigny for providing me with a proof of Theorem \ref{thm: regularpolyauto - nonarchdynamics - invariancemain} in the complex case. This research is partially supported by JSPS KAKENHI 26K16974 (Grant-in-Aid for Early-Career Scientists).

\subsection*{Use of AI}
The results are given by the author themselves. AI tools were used for checking the consistency of notation and for grammatical corrections.

\section{The Berkovich spaces} \label{sec: intro of berkoich spaces}
In this section, we recall basic results regarding the Berkovich geometry. 

Throughout this section, let $K$ be a field with complete, non-trivial, and non-archimedean valuation $|\cdot|$.

\subsection{Affine and Projective spaces} \label{subseq: intro of berkovich spaces - non-arch. space}

We construct the Berkovich affine and projective spaces $\bA^{N,\an}_K$ and $\bP^{N,\an}_K$, which are the so-called Berkovich analytification of the affine and projective spaces $\bA^N_K$ and $\bP^N_K$. 

For an affine space $\bA^N_K=\Spec K[x_1,\ldots,x_N]$, its Berkovich analytification is defined as follows: the underlying set consists of all the multiplicative seminorms over $K[x_1,\ldots,x_N]$ such that its restriction to the base field $K$ is the given valuation. The topology is the weakest one such that for any polynomial $f\in K[x_1,\ldots,x_N]$, the evaluation map $x\mapsto |f(x)|$ is continuous, where we denote the evaluation of $f$ at $x$ by $|f(x)|$. 

We note that, every maximal ideal $\fm$ of $K[x_1,\ldots,x_N]$ defines a point $x_{\fm}$ in $\bA^{N,\an}_K$ by
\begin{align*}
    |f(x_{\fm})|:=|f\ \text{mod}\ \fm|,
\end{align*}
where $f$ mod $\fm$ is an element of the finite field extension $L=K[x_1,\ldots,x_N]/\fm$ of $K$, to which the valuation $|\cdot|$on $K$ extends uniquely. Especially, each point $(z_1,\ldots,z_N)\in K^N$ defines a maximal ideal $(x_1-z_1,\ldots,x_N-z_N)$ in $K[x_1,\ldots,x_N]$, which coincides with the evaluation map at $(z_1,\ldots,z_N)$ as a seminorm. In this way, we obtain an inclusion $\Max(K[x_1,\ldots,x_N])\to \bA^{N,\an}_K$, which is known to be continuous and have dense image.

We note that there is a point so-called the Gauss point of $\bA^{N,\an}_K$, which is important in the later discussion. For any 
\begin{align*}
    f=\sum_{I}c_Ix^I,
\end{align*}
where $I=(i_1,\ldots,i_N)$ is a multi-index, we define
\begin{align*}
    |f(\zeta_G)|:=\max_{I}|c_I|.
\end{align*} 
By a simple calculation, this map defines a multiplicative norm on $K[x_1,\ldots,x_N]$. It is also known that if $K$ is algebraically closed,
\begin{align*}
    |f(\zeta_G)|=\max_{\|x\|\leq 1}|f(x)|,
\end{align*}
where $\|x\|=\max(|x_1|,\ldots,|x_N|)$ for any $x=(x_1,\ldots,x_N)\in K^N$.

The analytification of projective space is obtained by patching an affine cover together, but here, we present Baker-Rumely's "Proj" construction presented in \cite{BakerRumely2010}. In \cite{BakerRumely2010}, it is only defined for the projective line, but it is possible to extend the construction to higher dimensional projective spaces.

Take a homogeneous coordinate to write $\bP^{N}_K=\Proj K[X_0,\ldots, X_N]$. Let $S$ be the set of multiplicative seminorms on $K[X_0,\ldots,X_N]$ which is not identically zero on the maximal ideal $(X_0,\ldots,X_N)$. We introduce an equivalence relation $\sim$ as follows: Two seminorms $x$ and $y$ are equivalent if and only if there exists $C>0$ such that for any homogeneous polynomial $F\in K[X_0,\ldots,X_N]$ of degree $d$ such that $|F(x)|=C^d\cdot |F(y)|$. For instance, two points $x=(x_0\ldots,x_N)$ and $y=(y_0,\ldots,y_N)\in K^{N+1}$ defining the same point in $\bP^N(K)$, the multiplicative seminorms $x:F\mapsto |F(x_0,\ldots,x_N)|$ and $y:F\mapsto |F(y_0,\ldots,y_N)|$ are equivalent; indeed, we can take $C$ in the definition by $|s|$ where $s\in K$ is defined by $x_i=s\cdot y_i$ for any $i$.

For a multiplicative seminorm $x$ on $K[X_0,\ldots,X_N]$, we write the evaluation map at its equivalent class by $\llbracket \cdot(x) \rrbracket$. It is well-defined up to scalar multiplication if we restrict them to the set of homogeneous polynomials. We define the normalization of these classes. For the above $x$, let $i\in\{0,\ldots,N\}$ be an index such that $|X_i|$ attains $\max(|X_0|,\ldots,|X_N|)$. Since at least one of $|X_j|$ is nonzero, $|X_i|>0$, which enables us to extend $x$ to the seminorm on $K[X_0,\ldots,X_N,1/X_i]$ multiplicatively. Define
\begin{align*}
    \llbracket F(x) \rrbracket^*:= \llbracket f(x)\rrbracket,
\end{align*}
where $f=F(X_0/X_i,\ldots,X_{i-1}/X_i,1, X_{i+1}/X_i,\ldots, X_N/X_i)$. We note that this normalization defines a natural representative of the equivalence class of $S$; for any $x\sim y$, we have $\llbracket F(x)\rrbracket^*=\llbracket F(y)\rrbracket^*$ for any homogeneous $F$. Also, for any seminorm $x$ and any homogeneous polynomial $F$ of degree $d$, we have the following relation:
\begin{align*}
    \llbracket F(x)\rrbracket^*=\frac{\llbracket F(x)\rrbracket}{\max\left(\llbracket X_0(x)\rrbracket,\ldots, \llbracket X_N(x)\rrbracket\right)^d}.
\end{align*}
We define the analytic projective space $\bP^{N,\an}_K$ is defined by
\begin{align*}
    \bP^{N,\an}_K:=S/\sim.
\end{align*}
The topology is defined to be the weakest one such that $x\mapsto \llbracket F(x)\rrbracket^*$ is continuous for any homogeneous $F$.

Here we only describe the analytification of affine and projective spaces as topological spaces, though it actually has a structure of the sheaf of analytic functions. We omit details of them though we note that we can consider the notion of analytic functions, and also for an algebraic line bundle $L$, we can consider its analytification $L^{\an}$.

\subsection{Hybrid spaces} \label{subseq: intro of berkovich space - hybrid space}
We recall the definition of hybrid spaces, which is developed in a series of papers by Boucksom, Favre and Jonsson \cite{Sebastien:2015uz}, \cite{Sebastien:2017le} and imported to the study of complex dynamics by Favre \cite{Favre:2020rz}. It is a topological space whose shape looks like "a family of complex varieties degenerating to a non-archimedean variety." Here we present the definition when the family is the trivial one of the projective space $\pi:\bP^N(\bC)\times\bD^* \to \bD^*, (x,t)\mapsto t$ parametrized by the unit punctured disk $\bD^*$, or that of the affine spaces $\pi:\bC^N\times\bD^*\to\bD^*$, which are the only necessary cases in this paper.

Fix a real number $0<r<1$. The hybrid norm $|\cdot|_{\hyb}$ on the complex number field $\bC$ is defined by
\begin{align*}
    |x|_{\hyb}=
    \begin{cases}
    \max\{|x|,1\} &\text{ if }x\neq0, \\
    0 &\text{ if }x=0
    \end{cases}
\end{align*}
for any $x\in\bC$, where $|x|$ is the Euclidean norm of $x$. Given the hybrid norm on $\bC$, define the following ring $A_r$:
\begin{align*}
    A_r:=\left\{f(t)=\sum_{i=-\infty}^{\infty}a_it^i\ \middle|\ a_i\in\bC,\ \sum_{i=-\infty}^{\infty}|a_i|_{\hyb}r^i<\infty\right\}.
\end{align*}
We can naturally introduce a metric structure on $A_r$ as follows:
\begin{align*}
    \left|\sum_{i=-\infty}^{\infty}a_it^i\right|_{\hyb}:=\sum_{i=-\infty}^{\infty}|a_i|_{\hyb}r^i<\infty.
\end{align*}
We regard $A_r$ as a completely normed $\bC$-algebra with it, where the norm on $\bC$ is the hybrid norm. Then, by taking the Berkovich spectrum of $A_r$, we have an analytic space $\cC_{\hyb}(r)$. 

By definition, the points of the Berkovich spectrum of $A_r$ are bounded multiplicative seminorms on $A_r$. Hence, define the map $\tau:\bar{\mathbb{D}}_r\to \cC_{\hyb}(r)$ as follows:
\begin{align*}
    \tau(t)(f):=
    \begin{cases}
    |f(t)|^{\frac{\log r}{\log|t|}} &\text{ if }0<|t|\leq r, \\
    |f|_r=r^{\ord_0(f)} &\text{ if }t=0,
    \end{cases}
\end{align*}
for $f\in A_r$ and $t\in\bar{\mathbb{D}}_r$. It is known that $\tau$ is actually a homeomorphism.

The Berkovich analytification of the scheme $\Proj A_r[X_0,\ldots,X_N]$ (resp. $\Spec A_r[x_1,\ldots,x_N]$) is called the hybrid projective (resp. affine) $N$-space and is denoted by $\bP^{N,\hyb}$ (resp. $\bA^{N,\hyb}$).

It is defined similarly to $\bP^{N,\an}_{\bC((t))}$ though this time $A_r$ is not a non-archimedean field, which requires us to modify the definition slightly. As in $\bP^{N,\an}_{\bC((t))}$, we consider standard affine coverings and patch them together to define $\bP^{N,\hyb}$, so we only see the definition of the hybrid analytification of $\Spec A_r[x_1,\ldots,x_N]$, denoted by $\bA^{N,\hyb}$. The hybrid affine plane $\bA^{N,\hyb}$ is a locally compact Hausdorff space whose underlying set is 
\begin{align*}
    \bA^{N,\hyb}=\left\{x:A_r[x_1,\ldots,x_N]\to \bR_{\geq0}\quad\middle|\quad \begin{aligned}
   &x:\text{a multiplicative seminorm}, \\
   &x|_{A_r}\text{ is bounded by }|\cdot|_{\hyb}\end{aligned}\right\},
\end{align*}
with the topology of pointwise convergence. From \cite{Favre:2020rz}, we have the following properties:
\begin{itemize}
    \item The natural map $\pi_{\hyb}:\bP^{N,\hyb}\to\cC_{\hyb}(r)$ is continuous and proper.
    \item For any $t\neq0$, we have a canonical homeomorphism
    \begin{align*}
        \psi_t:\bP^N(\bC)\to\pi_{\hyb}^{-1}\{\tau(t)\}.
    \end{align*}
    \item For $t=0$, there is a canonical homeomorphism
    \begin{align*}
        \psi_0:\bP^{N,\an}_{\bC((t))}\to\pi_{\hyb}^{-1}\{\tau(0)\}.
    \end{align*}
\end{itemize}

Here, the map $\pi_{\hyb}$ is just a restriction map $x\to x|_{A_r}$. Any seminorm on $\cC_{\hyb}(r)$ induces a norm on $\bC$ by restriction, where by definition either the trivial norm or $|\cdot|^{\varepsilon}$ for $0<\varepsilon\leq 1$. We set $|\cdot|^0$ to be the trivial norm, which makes sense since $\{|\cdot|^{\varepsilon}\}_{\varepsilon>0}$ converges to the trivial norm as $\varepsilon\to0$. Thus, any $x\in\bP^{N,\hyb}$ is attached to a real number $\varepsilon\in[0,1]$ such that $x|_{\bC}=|\cdot|^{\varepsilon}$. We denote it by $\eta(x)$. The function $\eta:\bP^{N,\hyb}\to[0,1]$ is known to be continuous and proper as discussed in \cite{Favre:2020rz}.

\section{Monge-Amp{\`e}re measures on complex, non-archimedean, and hybrid spaces}\label{sec: regularpolyautom - Monge-Ampere}

\subsection{Pluripotential theory over Berkovich spaces}

We briefly recall the parts of non-Archimedean pluripotential theory developed by Chambert-Loir and Ducros \cite{ChambertLoirDucros2025} that will be used below. Let \(K\) be a complete non-Archimedean field, and let \(X\) be a good, paracompact, smooth \(K\)-analytic space without boundary. We write
\[
dd^c := d'd'',
\]
where $d'$ and $d''$ are the differential operators of forms/currents that appear in \cite{ChambertLoirDucros2025}, so that we can use the same notation in both the complex and non-archimedean settings.

The only plurisubharmonic functions that we shall use are continuous psh functions which are locally uniform limits of smooth psh functions. For such functions, Chambert-Loir and Ducros developed a Bedford--Taylor theory: if \(u_1,\dots,u_q\) are psh and \(T\) is a positive closed current, then the current
\[
dd^c u_1 \wedge \cdots \wedge dd^c u_q \wedge T
\]
is well-defined, symmetric in the functions \(u_i\), and continuous under local uniform convergence of the potentials. We shall also use the Poincar\'e--Lelong formula
\[
dd^c[\log|h|] = [\operatorname{div}(h)]
\]
for a regular meromorphic function \(h\). In particular, if \(h\) is invertible on an open subset, then \(\log|h|\) is pluriharmonic there.

The following special case is the only vanishing result that we shall need later.

\begin{prop} \label{prop: pluripotential theory - lth wedge of ddclogmax(1,|g1|,...,|gl-1|) is 0}
    Let $U$ be a good, paracompact and smooth $K$-analytic space without boundary, and $g_1,\ldots,g_{l-1}\in\Gamma(U,\cO_U)$ be a set of analytic functions. Then, we have
    \begin{align*}
        \left(\dd^c \log\max\left\{1,|g_1|,\ldots,|g_{l-1}|\right\}\right)^l=0
    \end{align*}
    on $U$.
\end{prop}
\begin{proof}
If $l>\dim(U)$, then the claim is trivial. Thus we may assume $l\leq \dim(U)$.
We argue by induction on $l$.

For $l=1$, we have
\[
\log \max\{1\} = 0,
\]
which is pluriharmonic, so the claim is clear.

Assume now that the statement holds for $l-1$. Set
\[
A := \log \max\{1, |g_1|, \ldots, |g_{l-2}|\},
\qquad
B := \max\{A,\log |g_{l-1}|\}.
\]
Then
\[
B=\log \max\{1, |g_1|, \ldots, |g_{l-1}|\}.
\]
We define two open subsets of $U$ by
\[
\Omega_1 := \{z\in U \mid \log |g_{l-1}(z)| < -1/4\},
\qquad
\Omega_2 := \{z\in U \mid \log |g_{l-1}(z)| > -3/4\}.
\]
Clearly, $\Omega_1\cup \Omega_2 = U$.

On $\Omega_1$, since $A\geq 0$ and $\log |g_{l-1}|<-1/4$, we have $B=A$. By the induction hypothesis,
\[
(\dd^c A)^{l-1}=0.
\]
Hence, by the Bedford--Taylor definition,
\[
(\dd^c B)^l = (\dd^c A)^l = \dd^c\bigl(A(\dd^c A)^{l-1}\bigr)=0
\]
on $\Omega_1$.

Next we work on $\Omega_2$. Since $g_{l-1}$ is invertible on $\Omega_2$, the function
$\log |g_{l-1}|$ is pluriharmonic there. Define
\[
A' := A-\log |g_{l-1}|,
\qquad
B' := B-\log |g_{l-1}| = \max\{A',0\}.
\]
Since $A'-A$ is pluriharmonic, we have
\[
(\dd^c A')^{l-1}=(\dd^c A)^{l-1}=0.
\]
Also, since $B-B'=\log |g_{l-1}|$ is pluriharmonic on $\Omega_2$, it is enough to show
\[
(\dd^c B')^l=0
\]
on $\Omega_2$.

Let $n:=\dim(U)$, and let $\varphi$ be any non-negative test form of type
$(n-l,n-l)$ with compact support contained in $\Omega_2$. For $\varepsilon>0$, set
\[
E_{\varepsilon}:=\{z\in \operatorname{Supp}(\varphi)\mid B'(z)\leq \varepsilon\}.
\]
Since $B'$ is plurisubharmonic, we have
\[
\int_{\Omega_2}\varphi\wedge (\dd^c B')^l \geq 0.
\]
By the integration by parts formula,
\[
\int_{\Omega_2}\varphi\wedge (\dd^c B')^l
=
\int_{\Omega_2} B'(\dd^c B')^{l-1}\wedge \dd^c\varphi.
\]
We split the right-hand side into
\[
\int_{\Omega_2} B'(\dd^c B')^{l-1}\wedge \dd^c\varphi
=
\int_{E_{\varepsilon}} B'(\dd^c B')^{l-1}\wedge \dd^c\varphi
+
\int_{\Omega_2\setminus E_{\varepsilon}} B'(\dd^c B')^{l-1}\wedge \dd^c\varphi.
\]

Now, on the open set $\{B'>0\}$ we have $B'=A'$. Since
$\operatorname{Supp}(\varphi)\setminus E_{\varepsilon}\subset \{B'>0\}$, by continuity there exists
an open neighborhood of $\operatorname{Supp}(\varphi)\setminus E_{\varepsilon}$ on which $B'=A'$.
Therefore, by locality of Bedford--Taylor products,
\[
B'(\dd^c B')^{l-1}=A'(\dd^c A')^{l-1}=0
\]
near $\operatorname{Supp}(\varphi)\setminus E_{\varepsilon}$, and hence
\[
\int_{\Omega_2\setminus E_{\varepsilon}} B'(\dd^c B')^{l-1}\wedge \dd^c\varphi = 0.
\]
Thus,
\[
\int_{\Omega_2}\varphi\wedge (\dd^c B')^l
=
\int_{E_{\varepsilon}} B'(\dd^c B')^{l-1}\wedge \dd^c\varphi.
\]

Since $0\leq B'\leq \varepsilon$ on $E_{\varepsilon}$, we obtain
\[
\left|
\int_{E_{\varepsilon}} B'(\dd^c B')^{l-1}\wedge \dd^c\varphi
\right|
\leq
\varepsilon
\int_{\operatorname{Supp}(\varphi)}
\left|(\dd^c B')^{l-1}\wedge \dd^c\varphi\right|.
\]
The right-hand side is independent of $\varepsilon$, so letting $\varepsilon\to 0$ gives
\[
\int_{\Omega_2}\varphi\wedge (\dd^c B')^l = 0.
\]
Since $\varphi$ was arbitrary, we conclude that
\[
(\dd^c B')^l=0
\]
on $\Omega_2$. Hence $(\dd^c B)^l=0$ on $\Omega_2$ as well.

Since $\Omega_1\cup\Omega_2=U$, we obtain
\[
(\dd^c B)^l=0
\]
on $U$. This proves the proposition.
\end{proof}

\subsection{Model functions} \label{subsection : backgrounds - model functions}
Here we briefly explain the notion of model functions, which is essential in the proof of our main theorem. Throughout this subsection, we consider a trivial family $X:=\bP^N(\bC)\times\bD^*$ and its natural model $\cX:=\bP^N(\bC)\times\bD$. We denote $L:=\cO(1)\times\bD^*$ and $\cL:=\cO(1)\times \bD$.

We set $\pi:X\to\bD$ and $\pi_{\cX}:\cX\to\bD$ to be the second projections, and 
\begin{align*}
    \bar{X}_r:=\pi^{-1}(\bar{\bD}_r^{*})=\bP^N(\bC)\times\bar{\bD}_r^*.
\end{align*}

A {\it regular admissible datum} $\cF$ consists of
\begin{itemize}
    \item a positive integer $d\in\bN$ called the {\it degree} of $\cF$;
    \item a vertical fractional ideal sheaf $\fA$ of $\cX$ such that $t^N \fA \subset \cO_{\cX}$ for large $N$;
    \item a finite set of meromorphic sections $\tau_0,\ldots\tau_l$ of $\cL^d$ defined in a neighborhood of $\bar{\cX}_r$ generating $\fA$;
    \item a log resolution $p:\cX'\to\cX$ of $\fA$.
\end{itemize}
Here, we say that the fractional ideal $\fA$ is vertical if the cosupport of $t^N\fA$ is in the locus $\{t=0\}$. To such a model $\cF$, we can attach the so-called {\it model functions}. We refer to three distinct but related functions, the ones defined over $\bar{X}_r$, $X^{\an}_{\bC((t))}$, and $X^{\hyb}$, by the same name here. They are all real-valued functions defined from a regular admissible datum, and in the end the model function over $X^{\hyb}$ is the most important. We construct the ones on $\bar{X}_r$ and $X_{\bC((t))}^{\an}$ first and then define the one on $X^{\hyb}$ by using the homeomorphism $\psi$.

we define the (archimedean) model function $\varphi_{\cF}:\bar{X}_r\to\bR$ by
\begin{align*}
    \varphi_{\cF}(x):=\log\max\{|\tau_0(x)|_*,\ldots,|\tau_l(x)|_*\}.
\end{align*}

We define the model function $g_{\cF}$ attached to a regular admissible datum $\cF$ on the non-archimedean space $X_{\bC((t))}^{\an}$ using local generators of $\fA$. First, take a homogeneous coordinate $X_0,\ldots,X_N$ and denote by $U_i$ an affine open neighborhood defined by $X_i\neq0$.

By Lemma 2.6 of \cite{Favre:2020rz}, there exist $f^{(i)},g_1^{(i)},\ldots,g_{l(i)}^{(i)}\in A_r$ such that $f^{(i)}\cdot\fA(U_i)$ is generated by $g_1^{(i)},\ldots,g_{l(i)}^{(i)}$. Then, for any $x\in U_i^{\hyb}$, define
\begin{align*}
    \log|\fA|(x)=\max\{\log|g_j^{(i)}(x)|,j=1,\ldots,l(i)\}-\log|f^{(i)}(x)|,
\end{align*}
where $|g_j^{(i)}(x)|$ and $|f^{(i)}(x)|$ denote the seminorms evaluated by $x$. This definition does not depend on the choice of generators, so that $\log|\fA|$ defines a real-valued continuous function on $X^{\an}_{\bC((t))}$ by restricting it to $\pi^{-1}\{\tau(0)\}$. We denote this function by $g_{\cF}$. 

With the two model functions attached to $\cF$, we define the model function $\Phi_{\cF}$ on the so-called hybrid spaces, denoted by $X^{\hyb}$. Define the model function $\Phi_{\cF}:X^{\hyb}\to\bR$ by
\begin{align*}
    \Phi_{\cF}(x):=
    \begin{cases}
        \eta(x)\cdot \varphi_{\cF}(\psi^{-1}(x)) &\text{ if }x\in\pi_{\hyb}^{-1}(\tau(\bar{\bD}^*_r)),\\
        g_{\cF}(x) &\text{ otherwise},
    \end{cases}
\end{align*}
where we defined the function $\eta$ in Section \ref{subseq: intro of berkovich space - hybrid space}. By Theorem 2.10 of \cite{Favre:2020rz}, $\Phi_{\cF}$ is continuous.

An important fact for model functions is that they form a dense subset of all the continuous functions on the hybrid space $X^{\hyb}$. Namely, the following result is shown by Favre:
\begin{prop}[\cite{Favre:2020rz}, Theorem 2.12]\label{proposition; hybridK3 - backgrounds - model functions - density}
    Let $\cD(X^{\hyb})$ be the space of all the functions of the form $q\Phi_{\cF}-q'\Phi_{\cF'}$ where $\cF$ and $\cF'$ are regular admissible data and $q$ and $q'$ are positive rational numbers so that $q\deg(\cF)=q'\deg(\cF')$. Then, $\cD(X^{\hyb})$ is a $\bQ$-vector space that is dense in the space of real-valued continuous functions on $X^{\hyb}$ with the topology of uniform convergence.
\end{prop}

\subsection{Monge-Amp{\`e}re measures}
In the same notation as the previous subsection, we consider ample line bundles $L_1,\ldots, L_n$ and their semiample models $\cL_1,\ldots,\cL_n$. We see in Section \ref{subseq: intro of berkovich spaces - non-arch. space} that $X$ and $L_i$ induce a Berkovich variety $X^{\an}_{\bC((t))}$ and its line bundle $L^{\an}_{i,\bC((t))}$ respectively. In this subsection, we define the mixed Monge-Amp{\`e}re measures both on $\cX$ and $X_{\bC((t))}^{\an}$

First, we consider the complex analytic case. Let $h_1,\ldots,h_n$ be smooth metrics on $\cL_1,\ldots,\cL_n$ respectively. Each $h_i$ defines a curvature form denoted by $dd^c \log h_i$, which is locally defined by $dd^c \log h_i\circ s$ for any local section $s$ of $L_i$, respectively. Here, we write
\begin{align*}
    d^c=\frac{1}{2\pi i}\left(\partial-\bar{\partial}\right),
\end{align*}
following Demaily \cite{Demailly:1993if}.

We note that curvature forms are originally defined for smooth metrics, but they extend to continuous psh metrics using currents, thanks to Bedford and Taylor \cite{BedfordTaylor1976}. That is if we assume $h_i$'s are continuous psh metrics, i.e., $\log h_i\circ s$ is psh for any local section $s$ of $L_i$, the local Laplacian $dd^c \log h_i\circ s$ still makes sense if we regard $\log h_i\circ s$ as a current, and it is independent of the choice of $s$. That is, they can be patched together to define a global current $dd^c \log h_i$. It is then a closed positive (1,1)-current on $X$. Moreover, since they have locally continuous potentials, we take their wedge product to define a $(n,n)$-form $dd^c \log h_1 \wedge \cdots\wedge dd^c\log h_n$.

For any $t \in \bD$, define  $[X_t]:=dd^c\log|\pi_{\cX}-t|$. For a continuous psh metric $h_1,\ldots,h_n$ on $\cL_1,\ldots,\cL_n$ respectively, we define the mixed Monge-Amp{\`e}re measure associated to $h_1\ldots,h_n$ at $t\in\bD$ by
\begin{align*}
    \MA_t(h_1,\ldots,h_n):=dd^c \log h_1 \wedge\cdots\wedge dd^c \log h_n \wedge [X_t].
\end{align*}
We note that for $t\in\bD^*$, $dd^c\log|\pi_{\cX}-t|$ is the current of integration over $\pi^{-1}\{t\}$, while for $t=0$,
\begin{align*}
    [X_0]=\sum_{E\subset \cX_0}b_E[E],
\end{align*}
where $E$ ranges over all the irreducible components of $\cX_0$ and $b_E$ is the order of zeros of $\pi^*t$ at the generic point of $E$. 

\begin{prop}[\cite{Demailly:1993if}, Corollary 1.6]\label{proposition: backgrounds - continuity of MA measures}
    Let $U\subset X$ be a domain and $\{u_1^k\}_{k=1}^{\infty},\ldots,\{u_n^k\}_{k=1}^{\infty}$ be sequences of continuous psh functions converging locally uniformly to $u_1,\ldots,u_n$ respectively. Then, we have a weak convergence of measures
    \begin{align*}
        \lim_{k\to\infty}dd^c u_1^k\wedge\cdots\wedge dd^cu_n^k=dd^c u_1\wedge\cdots\wedge dd^cu_n.
    \end{align*}
\end{prop}

On the other hand, we have a similar theory of currents and psh functions over non-archimedean varieties due to Chambert-Loir and Ducros \cite{ChambertLoirDucros2025}. Since we only use the mixed Monge-Amp{\`e}ere measure for model metrics in the later discussion, we give the exact formula for them as a definition.

We consider models $\cL_1,\ldots,\cL_n$ of $L_1,\ldots,L_n$ over some snc model $\cX$ respectively. Each model $\cL_i$ defines a metric on $L_{i,\bC((t))}^{\an}$, which is called the model metric $g_{\cL_i}$ attached to $\cL_i$. We will not give a definition here and refer to Section 5.3 of \cite{BoucksomEriksson2021}. For such metrics, the mixed Monge-Amp{\`e}ere measure, denoted by $\MA_{\NA}(g_{\cL_1},\ldots,g_{\cL_n})$, is defined as an atomic measure
\begin{align*}
    \MA_{\NA}(g_{\cL_1},\ldots,g_{\cL_n})=\sum_{E\subset\cX_0}(\cL_1|_E\cdot\cdots\cdot\cL_n|_E)\delta_{x_E},
\end{align*}
where $E$ runs over all the irreducible components of $\cX_0$ and $\delta_{x_E}$ denotes a Dirac mass at the divisorial point $x_E$ corresponding to $E$.

For both archimedean and non-archimedean mixed Monge Amp{\`e}re measures, the following continuity is known \cite{BoucksomEriksson2021}:
\begin{prop}[\cite{BoucksomEriksson2021},Theorem 8.4]
    Let $h_i$ be a metric on $\cL_i$ and $g_i$ be a metric on $L_{i,\bC((t))}^{\an}$. we assume that there exist sequences of metrics $\{h_i^{(j)}\}_{j=1}^{\infty}$ on $\cL_i$ and $\{g_i^{(j)}\}_{j=1}^{\infty}$ on $L_{i,\bC((t))}^{\an}$ such that
    \begin{itemize}
        \item each $g_i^{(j)}$ is a model metric for any $i$ and $j$;
        \item $h_i^{(j)}\to h_i$ and $g_i^{(j)}\to g_i$ respectively as $j\to\infty$, and the convergence is uniform for every $i$.
    \end{itemize}
    Then, $\MA_{\NA}(g_1,\ldots,g_n)$ is well-defined as a positive measure on $X_{\bC((t))}^{\an}$ and we have the following weak convergence of measures:
    \begin{align*}
        &\lim_{j\to\infty}\MA_t(h_1^{(j)},\ldots,h_n^{(j)})=\MA_t(h_1,\ldots,h_n);\\
        &\lim_{j\to\infty}\MA_{\NA}(g_1^{(j)},\ldots,g_n^{(j)})=\MA_{\NA}(g_1,\ldots,g_n)
    \end{align*}
\end{prop}

We remark that Theorem 8.4 of \cite{BoucksomEriksson2021} states a stronger result than what we present in the non-archimedean case. Namely, they show that the mixed Monge-Amp{\`e}re operator is extended continuously to the set of continuous psh-regularizable metrics. We do not present details about such metrics. However, we remark that by Theorem 7.14 and Definition 7.2 of \cite{BoucksomEriksson2021}, metrics that can be written as a uniform limit of model metrics are psh-regularizable.

\subsection{Monge-Amp\`ere measure on the hybrid projective space}

In this section, we only treat the case $\cL_1=\cdots=\cL_n=\cL=\cO(1)\times \bD$. We fix a homogeneous coordinate and its Fubini-Study metric $\|\cdot\|_*$ on $\cL$. As in the previous section, we equip the model metric on $L^{\an}_{\bC((t))}$ induced from $\cL$. In this way we identify a function $\phi$ on $\bP^N(\bC)\times\bD\*$ or on $\bP^{N,\an}_{\bC((t))}$ with a metric $\|\cdot\|_*\cdot \exp(\phi)$. We denote
\begin{align*}
    &\MA_t(\phi):=\MA_t(\phi,\ldots,\phi); \\
    &\MA_{NA}(g):=\MA_{\NA}(g,\ldots,g).
\end{align*}

For any regular admissible datum $\cF$, define a measure $\mu_{t,\cF}$ for $t\in\bar{\bD}_r$ on $\pi_{\hyb}^{-1}(\tau(t))\subset\bP^{N,\hyb}$ as follows:
\begin{align*}
    \mu_{t,\cF}&=(\psi_t)_*(\MA_t(\varphi_{\cF}))\text{ for }t\neq0, \\
    \mu_{0,\cF}&=(\psi_0)_*(\MA_{\NA}(g_{\cF}))\text{ for }t=0.
\end{align*}

For the family of measures $\{\mu_{t,\cF}\}_{t\in\bar{\bD}_r}$ on $\bP^{N,\hyb}$, we have a weak convergence
\begin{align*}
    \lim_{t\to0, t\neq0}\mu_{t,\cF}=\mu_{0,\cF}.
\end{align*}

The most important fact from \cite{Favre:2020rz} is that the same weak convergence also follows for a uniform function;

\begin{prop}[\cite{Favre:2020rz}, Definition 4.1 \& Theorem 4.2] \label{prop:themainfact}

Let $\phi:\bP^N(\bC)\times\bD^*\to\bR$ be a uniform function; that is, there exists a sequence of regular admissible data $\{\cF_n\}_n$ of degree $d_n$ for each $n$ and a decreasing sequence of positive real numbers $\{\epsilon_n\}_n$ converging to $0$ such that, for any $t$ satisfying $0<|t|\leq r$,
\begin{align*}
     \sup_{z\in \bP^N(\bC)\times\{t\}}\left|\phi(z)-\frac{1}{d_n}\phi_{\cF_n}(z)\right|\leq \epsilon_n \log|t|^{-1}.
\end{align*}
Then, the sequence of functions $\{g_{\cF_n}/d_n\}_n$ on $\bP^{N,\an}_{\bC((t))}$ uniformly converges to a continuous function $g$, and we can define the continuous function $\Phi$ such that
\begin{align*}
    \Phi(z)=\begin{cases}
        \phi(\psi_t^{-1}(z)) &\text{ if }z\in\pi_{\hyb}^{-1}\{\tau(\bD^*)\};\\
        g(\psi_0^{-1}(z)) &\text{ otherwise.}
    \end{cases}
\end{align*}
Moreover, by setting
\begin{align*}
    \MA_{t,\hyb}(\Phi)&:=(\psi_t)_*\MA_t(\phi) \text{ for }t\neq0,\\
    \MA_{0,\hyb}(\Phi)&:=(\psi_0)_*\MA_0(g) \text{ for }t=0,
\end{align*}
we obtain the following weak convergence:
\begin{align*}
    \lim_{t\to0}\MA_{t,\hyb}(\Phi)=\MA_{0,\hyb}(\Phi).
\end{align*}
\end{prop}

\section{Invariant measure of the non-archimedean dynamics}\label{sec: regularpolyautom - Invariant measure of the non-archimedean dynamics}

We study dynamical systems of regular polynomial automorphisms over affine spaces over non-archimedean fields, which will be used in the later discussion. There are several known results for such dynamics as in \cite{Shu:2013ma} and \cite{Lee:2012mb}, and our work is an extension of them. Before we show our result, we first fix notation and recall previous works, mainly in \cite{Shu:2013ma}. We remark that we use same symbols such as $t$ in the previous sections for different meanings, which we only use in this section. 

\subsection{Regular polynomial automorphisms}

Let $K$ be a field of characteristic zero with complete, non-trivial, and non-archimedean valuation. We fix $x=(x_1,\ldots,x_N)$ as an inhomogeneous coordinate of $\bA^N_K$, and consider the projective space $\bP^N_K$ containing the $\bA^N_K$ with coordinate $(X:S)$ such that
\begin{align*}
    x_i=X_i/S,
\end{align*}
and thus the line of infinity is defined by the equation $S=0$. Then, any polynomial map $f=(f_1,\ldots,f_N):\bA^N_K\to\bA^N_K$ of degree $d$ is extended to a rational map $F=(F_1:\cdots:F_N:S^d):\bP^N_K\dashedrightarrow\bP^N_K$, where the $F_i$ are homogenizations of the $f_i$. As in the complex dynamics, a polynomial automorphism $f:\bA^{N}_K\to\bA^{N}_K$ is said to be {\it regular} if the intersection of the indeterminacy loci of $f$ and $f^{-1}$ when they are extended to the projective space $\bP^N_K$ is empty. For such $f$, we denote the extension to $\bP^N_K$ of $f^n$ by $F^n$ for $n\in \bZ$. In this case, Sibony \cite{Sibony99} shows the following:
\begin{prop}\label{prop: regularpolyautom - fact - l+ + l_- =N, d^l=d^l}
    Let $K$, $f$,$F$ and $F^{-1}$ be as above. we put $d_{\pm},l_{\pm}$ as follows:
    \begin{align*}
        d_{\pm}&:=\deg f^{\pm 1}; \\
        l_+ &:= \dim I(F) +1;\\
        l_- &:= \dim I(F^{-1}) +1.
    \end{align*}
    Then, we have the following relation on these numbers:
    \begin{align*}
        &l_++l_-=N, \text{ and } \\
        &d_+^{l_-}=d_-^{l_+}.
    \end{align*}
\end{prop}

Sibony studied such dynamical systems \cite{Sibony99} in the complex setting, i.e., for the same polynomial automorphisms but with coefficients in $\bC$. In particular, they constructed the following escape-rate functions:
\begin{align*}
    &G^+(x)=\lim_{n\to\infty}\frac{1}{d_+^n}\log^+\|f^n(x)\|;\\
    &G^-(x)=\lim_{n\to\infty}\frac{1}{d_-^n}\log^+\|f^{-n}(x)\|,
\end{align*}
where $\|\cdot\|$ is any norm on $\bC^N$. For the later discussion, we also 
put 
\begin{align*}
    G(x):=\max\left\{G^+(x),G^-(x)\right\}.
\end{align*}
They show that $G^{\pm}$ are well-defined, continuous pluri-subharmonic functions on $\bC^N$. They also define (un-)stable currents
\begin{align*}
    T_{\pm}:=\dd^c G^{\pm},
\end{align*}
and show the following properties of these currents:

\begin{prop}[\cite{Sibony99}, 1.10]
    Let $f:\bC^N\to\bC^N$ be a regular polynomial automorphism and $d_{\pm}$ and $l_{\pm}$ as above. Then, there exists positive and closed (1,1)-currents $T_+$ and $T_-$ such that $\mu:=T_+^{\wedge l_-}\wedge T_-^{\wedge l_+}$ is a positive measure invariant under the pull-back action of $f$.
\end{prop}

To show the hybrid convergence of families of regular polynomial automorphisms, we need an analogous result in the non-archimedean setting. 

We note that since we consider $\bC((t))$, a field of characteristic $0$, the algebraic relation, such as $l_++l_-=N$ and $d_+^{l_-}=d_-^{l_+}$, still holds in the non-archimedean situation.

The non-archimedean version of the above functions and measures are studied by Kawaguchi \cite{Shu:2013ma} and Lee \cite{Lee:2012mb}, and in particular, the dynamical Green functions $G_{\NA}^{\pm}$ have already been constructed:

\begin{prop}\label{prop: regularpolyautom - NA dynamics - local unif conv of G}
    Let $f:\bA^{N,\an}_K\to \bA^{N,\an}_K$ be a regular polynomial automorphism and $F,G,d_{\pm}$ be as above. The functions
    \begin{align*}
        G^{\pm}_{\NA}=\lim_{n\to\infty}\frac{1}{d_{\pm}^n}\log^+\|f^{\pm n}\|
    \end{align*}
    are well-defined continuous functions with the limit uniform on each compact set of $\bA^{N,\an}_K$.
\end{prop}

We note that the well-definedness is already shown by Kawaguchi for the set of the classical points and by Lee for the whole space $\bA^{N,\an}_K$. Lee also showed the local uniform convergence of $G^{+}_{\NA}$ (resp. $G^-_{\NA}$) on the set of point with forward (resp. backward) unbounded orbits. Therefore it is enough to show the uniform convergence around $K^{\pm}_f$. We temporarily assume the claim is true until stating the main lemma, Proposition \ref{prop: regularpolyautom - NAdynamics - Kawaguchi's results}.

Note that, for each $n$, the function $\log^+\|f^{\pm n}(x)\|/d_{\pm}^n$ is plurisubharmonic in the sense of Chambert-Loir and Ducros \cite{ChambertLoirDucros2025}, hence so is $G_{\NA}^{\pm}$. By taking $\dd^c$ of these functions, we obtain the two positive closed (1,1)-currents
\begin{align*}
    T_{\pm,\NA}:=\dd^c G^{\pm}_{\NA}.
\end{align*} 
By construction, we see that $(f^{\pm 1})^*T_{\pm,\NA}=d_{\pm}^{-1}\cdot T_{\pm,\NA}$.

In this section, we show the following proposition, which is necessary to show our main result:
\begin{thm}\label{thm: regularpolyauto - nonarchdynamics - invariancemain}
    As above, let $f:\bA^N_K\to\bA^N_K$ be a regular polynomial automorphism of degree $d_+$ with $\deg f^{-1}=d_-$, $l_{\pm}:=\dim I(F^{\pm 1})+1$. Set
    \begin{align*}
        G_{\NA}^{\pm}(x):=\lim_{n\to\infty}\frac{1}{d^n_{\pm}}\log^+\|f^{\pm n}(x)\|,
    \end{align*}
    and $T_{\pm,\NA}:=\dd^c G_{\NA}^{\pm}$. Then, the positive measure $\mu_{\NA}:=T_{+,\NA}^{\wedge l_-}\wedge T_{-,\NA}^{\wedge l_+}$ is invariant under the action $f^*$ of the pull-back by the map $f$. Moreover, the following equality of measures on $\bA^{N,\an}_{K}$ holds:
    \begin{align*}
        \mu_{\NA}=\left(\dd^c\max\{G_{\NA}^{+},G_{\NA}^{-}\}\right)^N.
    \end{align*}
\end{thm}

We remark that, in \cite{irokawa2023hybrid}, the author already showed the main theorem of this paper in the case of H\'{e}non mappings over $\bC^2$, but did not show Theorem \ref{thm: regularpolyauto - nonarchdynamics - invariancemain}. The main result can be shown by the techniques of hybrid spaces by Boucksom-Favre-Jonsson \cite{Sebastien:2015uz} and  \cite{Sebastien:2017le}, though as we will see in the rest of the section, this theorem requires local arguments of $\bA^{N,\an}_K$, which is impossible by their theory and here we essentially need to use non-archimedean pluripotential theory by Chambert-Loir-Ducros \cite{ChambertLoirDucros2025}.

\subsection{Non-archimedean filled Julia sets}\label{subsec: regularpolyautom - NA dynamics - Julia sets}

In this subsection, we study some properties of forward and backward filled Julia sets $K_f^{\pm}$. Especially, the goal of this subsection, the key result for Theorem \ref{thm: regularpolyauto - nonarchdynamics - invariancemain} is as follows:
\begin{prop} \label{prop: regularpolyauto - nonarchdynamics - GreensoutsideofJuliaset}
    In the same notation as Theorem \ref{thm: regularpolyauto - nonarchdynamics - invariancemain}, we set 
    \begin{align*}
        K^{\pm}_{f}:=\left\{z\in\bA^{N,\an}_K \middle| \|f^{\pm n}(z)\|\not\to\infty\right\},
    \end{align*}
    and $\overline{K}^{\pm}_f$ to be its closure in $\bP^{N,\an}_K$. Then, we have the following relation:
    \begin{align*}
        \left(\dd^c G_{\NA}^{\pm}\right)^{l_{\mp}}=0 \quad \text{on} \quad\bA^{N,\an}_K\setminus K^{\pm}_f.
    \end{align*}
\end{prop}

Following Kawaguchi \cite{Shu:2013ma}, we prepare several notations.

We consider the homogeneous coordinate $[X_1:\cdots:X_N:S]$ of $\bP^N$ same as the previous subsection. Let $F=[F_1:\cdots:F_N:S^{d_+}]:\bP^N\dashrightarrow\bP^N$ be a homogeneous lift of $f$ and $G=[G_1:\cdots:G_N:S^{d_-}]$ be that of $f^{-1}$. By assumption, $F_i$ (resp. $G_i$) for any $i=1,\ldots,N$ is a homogeneous polynomial of degree $d_+$ (resp. $d_-$) with coefficients in $K$, and $\{F_i,G_j\}_{i,j=1}^{N}$ have no common zero other than $X_1=\cdots=X_N=S=0$. By well-known arguments about resultants, there exist $P_{ij},Q_{ij}\in K[X_1,\ldots,X_N]$ such that
We write $F^+:=F$ and $F^-:=G$.
\begin{align*}
    \sum_{j=1}^N P_{ij}(X)F_j(X_1,\ldots,X_N,0)+\sum_{j=1}^N Q_{ij}(X)G_j(X_1,\ldots,X_N,0)=X_i^m
\end{align*}
for any $i=1,\ldots,N$. Then, since we can write
\begin{align*}
    F_i(X_1,\ldots,X_N,S)=F_i(X_1,\ldots,X_N,0)+SF_i'(X_1,\ldots,X_N,S),
\end{align*}
there exists $R_i(X,S)\in K[X_1,\ldots,X_N,S]$ such that
\begin{align*}
    \sum_{j=1}^N P_{ij}(X)F_j(X_1,\ldots,X_N,S)+\sum_{j=1}^N Q_{ij}(X)G_j(X_1,\ldots,X_N,S)+SR_i(X_1,\ldots,X_N,S)=X_i^m
\end{align*}
for each $i=1,\ldots,N$. By taking homogeneous part, we may assume $P_{ij},Q_{ij}, R_i$ are all homogeneous. Let $C'$ be the maximum of absolute values of all the coefficients of $P_{ij},Q_{ij}, R_i$, and $C:=\max\{1,C'\}$. Take small $\delta$ and $\varepsilon$ so that
\begin{align*}
    &\delta<\frac{1}{C}; \\
    &\varepsilon<\min\left\{\frac{\delta^{1/d_+}}{C},\frac{\delta^{1/d_-}}{C}\right\}; \\
    &\varepsilon^{d_{\pm}-1}<\delta.
\end{align*}

Using these numbers $\varepsilon$ and $\delta$, we define the following sets:
\begin{align*}
    &N_{\delta,\varepsilon}^{\pm}:=\left\{ z\in\bA^{N,\an}_K \middle| 1<\varepsilon\|z\| \text{ and } \|f^{\pm1}(z)\|<\delta\|z\|^{d_{\pm}} \right\}; \\
    &V_{\delta,\varepsilon}^{\pm}:=\bA^{N,\an}_K\setminus N_{\delta,\varepsilon}^{\pm}; \\
    &U_{\delta,\varepsilon}^{\pm}:=\left\{z\in\bA^{N,\an}_K\middle| 1<\varepsilon\|z\|\text{ and } \|f^{\pm1}(z)\|\geq \delta\|z\|^{d_{\pm}}\right\}
\end{align*}
These three sets are defined on $\bA^{N,\an}_K$. We can naturally extend them to $\bP^{N,\an}_K$ as follows: as above, we fix $\bP^N_K=\Proj K[X_1,\ldots,X_N,S]$. Since any point $x\in \bP^{N,\an}_K$ defines an equivalence class of the seminorms as we saw in Section \ref{subseq: intro of berkovich spaces - non-arch. space}. We put 
\begin{align*}
&\|X(z)\|:=\max(\llbracket X_1(z)\rrbracket,\ldots,\llbracket X_N(z)\rrbracket);  \\
&\|S(z),X(z)\|:=\max(\llbracket S(z)\rrbracket,\llbracket X_1(z)\rrbracket,\ldots,\llbracket X_N(z)\rrbracket), 
\end{align*}
which are defined up to scalar multiplication. With these notations,
\begin{align*}
    &\tilde{N}_{\delta,\varepsilon}^{\pm}:=\left\{ z\in\bP^{N,\an}_K \middle| \llbracket S(z)\rrbracket<\varepsilon\|X(z)\| \text{ and } \|F^{\pm}(z)\|<\delta\|S(z),X(z)\|^{d_{\pm}} \right\}; \\
    &\tilde{V}_{\delta,\varepsilon}^{\pm}:=\bP^{N,\an}_K\setminus \tilde{N}_{\delta,\varepsilon}^{\pm}; \\
    &\tilde{U}_{\delta,\varepsilon}^{\pm}:=\left\{z\in\bP^{N,\an}_K\middle| \|X(z)\|>\llbracket S(z)\rrbracket/\varepsilon\text{ and } \|F^{\pm}(z)\|\geq \delta\|S(z),X(z)\|^{d_{\pm}}\right\}.
\end{align*}
\begin{prop}[\cite{Shu:2013ma}, Proposition 3.5]\label{prop: regularpolyautom - NAdynamics - Kawaguchi's results}
    In the above notation, the following relations hold:
    \begin{itemize}
        \item $f(U_{\delta,\varepsilon}^{+})\subset U_{\delta,\varepsilon}^+$ and $f(V_{\delta,\varepsilon}^{+})\subset V_{\delta,\varepsilon}^+$;
        \item $\tilde{U}_{\delta,\varepsilon}^+\cap I(F)=\emptyset$;
        \item $\bigcup_{n\geq0} f^{-n}(U_{\delta,\varepsilon}^+)=\bA^{N,\an}_K\setminus K_f^+$. 
    \end{itemize}
    Also, the same statements for $f^{-1}$ and $G$ also hold if we replace $\tilde{U}_{\delta,\varepsilon}^+$ and $K_f^+$ by $\tilde{U}_{\delta,\varepsilon}^-$ and $K_f^-$ respectively.
\end{prop}

Before showing Proposition \ref{prop: regularpolyauto - nonarchdynamics - GreensoutsideofJuliaset}, we first show Proposition \ref{prop: regularpolyautom - NA dynamics - local unif conv of G} from the above Proposition.

The proof is parallel to the same statement for complex regular polynomial automorophisms, and same as the arguments used by the author in \cite{irokawa2023hybrid} for the dynamics of non-archimedean H\'enon maps.

\begin{proof}[Proof of Proposition \ref{prop: regularpolyautom - NA dynamics - local unif conv of G}]
    We note that $K_f^+$ is the set of zeros of $G_{\NA}^+$, that is, the continuity on the interior of $K^+_f$ is trivial. Also by Lee, the uniform convergence, and hence the continuity, on $U_{\delta,\varepsilon}^+$ is already shown. 

    First we show that the continuity of $G^+_{\NA}$ around the boundary of $K^+_f$ denoted by $J_f^+$.

    Fix any point $z\in J_f^+$ and take any net $(z_{\lambda})\to z$ where $\lambda\in \Lambda$ is a directed set. Since $G^+_{\NA}=0$ on $K_f^+$, we may assume $(z_{\lambda})\subset \bA^{N,\an}_K\setminus K_f^+$. Also by $\overline{V}^+_{\delta,\varepsilon}\cap J_f^+=\emptyset$, we may also assume $(z_{\lambda})\subset\bA^{N,\an}_K\setminus U^+_{\delta,\varepsilon}.$ For any $\lambda$, we set
    \begin{align*}
        n(\lambda):=\min\left\{n\middle| f^n(z_{\lambda})\in U_{\delta,\varepsilon}^+\right\}.
    \end{align*}
    Since $f^n(z)\not\in V_{\delta,\varepsilon}^+$ for any $n$ and $z_{\lambda}\to z$, 
    \begin{align*}
        \lim_{\lambda\in\Lambda}n(\lambda)=\infty.
    \end{align*}
    Now we take any $e>0$.\footnote{We do a usual $\varepsilon-N$ argument here, though we use $e$ as the error term as $\varepsilon$ is already used.} Take $m=m(e)$ so large that
    \begin{align*}
        \frac{\log \varepsilon^{-1}}{d^m}+\sum_{i>m}\frac{\log\delta^{-1}}{d^i}<e,
    \end{align*}
    then there exists $\lambda_0\in\Lambda$ such that for any $\lambda\geq \lambda_0$, $n(\lambda)>m.$ For such $\lambda$, we obtain
    \begin{align*}
        \log^+\|f^{n(\lambda)}(z_{\lambda})\|\leq\log\varepsilon^{-1},
    \end{align*}
    that is,
    \begin{align*}
        G_{\NA}^+(z_{\lambda})&=\frac{1}{d^{n(\lambda)}}G_{\NA}^+(f^{n(\lambda)}(z_{\lambda})) \\
        &<\frac{1}{d^{n(\lambda})}\log^+\|f^{n(\lambda)}(z_{\lambda})\|+\sum_{i>n(\lambda)}\frac{\log\delta^{-1}}{d^i} \\
        &<\frac{\log \varepsilon^{-1}}{d^m}+\sum_{i>m}\frac{\log\delta^{-1}}{d^i} \\
        &<e,
    \end{align*}
    which shows the continuity.

    Now we show the uniform convergence. We set $B(\varepsilon)$ to be the closed disk of $\bA^{N,\an}_K$ centered at origin with radius $\varepsilon^{-1}$. Since $\{B(\varepsilon)\}$ covers the whole $\bA^{N,\an}_K$, it is enough to show that 
    \begin{align*}
        G_{n,\NA}^+(z):=\frac{1}{d^n}\log^+\|f^n(z)\|
    \end{align*}
    converges uniformly to $G_{NA}^+$ on $B_{\varepsilon}$ as $n\to \infty$. We note that 
    \begin{align*}
        V^+_{\delta,\varepsilon}=B(\varepsilon)\cup U^+_{\delta,\varepsilon}
    \end{align*}
    and by Proposition \ref{prop: regularpolyautom - NAdynamics - Kawaguchi's results}, $f(B(\varepsilon))\subset V^+_{\delta,\varepsilon}$.
    To show that, we take any $e>0$ and show that there exists $m=m(e)$ such that
    \begin{align*}
        \left|G_{\NA}^+(z)-G_{n,\NA}^+(z)\right|<e
    \end{align*}
    for any $n>m$ and $z\in B(\varepsilon)$. To do that, set 
    \begin{align*}
        n(z):=\min\{n|f^n(z)\in U_{\delta,\varepsilon}^+\}.
    \end{align*}
    If $z$ is of bounded orbit, set $n(z)=\infty$. Take $M$ so large that
    \begin{align*}
        \frac{1}{d^M}\left(\log \varepsilon^{-1}+\log C+\sum_{i\geq1}\frac{\log\delta^{-1}}{d^i}\right)<2\varepsilon.
    \end{align*}
    Then, for $n>\max(M,n(z))$, we have
    \begin{align*}
        \left|G_{\NA}^+(z)-G_{n,\NA}^+(z)\right|&<\frac{1}{d^{n(z)}}\left|G^+_{\NA}\left(f^{n(z)}(z)\right)-G^+_{n,\NA}\left(f^{n(z)}(z)\right)\right| \\
        &<\frac{1}{d^{n(z)}}\sum_{i\geq 1}\frac{\log\delta^{-1}}{d^i}<e.
    \end{align*}
    Thus, it is enough to show the same equality for $M\leq n<n(z)$. For such $n$, we have by definition $f^n(z)\in B(\varepsilon)$. Hence
    \begin{align*}
        0<\frac{1}{d^n}\log|f^n(z)|<\frac{1}{d^n}\log\varepsilon^{-1}.
    \end{align*}
    for $G_{\NA}^+$,
    \begin{align*}
        0&<G_{\NA}^+(z)=\frac{1}{d^{n(z)}}G^+\left(f^{n(z)}(z)\right) \\
        &<\frac{1}{d^{n(z)}}\left(\log\left|f(f^{n(z)-1}(z))\right|+\sum_{i\geq1}\frac{\log\delta^{-1}}{d^i}\right) \\
        &<\frac{1}{d^M}\left(\log\varepsilon^{-1}+\log C+\sum_{i\geq 1}\frac{\log\delta^{-1}}{d^i}\right).
    \end{align*}
    We put them together to obtain
    \begin{align*}
        \left|G_{\NA}^+(z)-G_{n,\NA}^+(z)\right|<\frac{2}{d^M}\log\varepsilon^{-1}+\frac{1}{d^M}\left(\log C+\sum_{i\geq1}\frac{\log\delta^{-1}}{d^i}\right)<e.
    \end{align*}

    \end{proof}

Now, we go back to the proof Proposition \ref{prop: regularpolyauto - nonarchdynamics - GreensoutsideofJuliaset}. We prepare several technical lemmas. 

\begin{prop}\label{prop: regularpolyauto - nonarchdynamics - LimitOfUnboundedPoints}
    Let $f:\bA^{n,\an}_K\to\bA^{n,\an}_K$ be a regular polynomial automorphism, and $F^n \colon \mathbb P_K^{N,\an} \dashrightarrow \mathbb P_K^{N,\an}$ be the extension of $f^n$ to the projective space for any $n\in\bZ$. We suppose $\{z_1=\cdots=z_{l_-}\}\cap I(F^{-1})=\emptyset$.
    For a positive number $A>0$, we set
    \begin{align*}
        W_A:=\left\{z\in\bA^{N,\an}_K\middle|\max\left(\left|z_{l_-+1}\right|,\ldots,\left|z_{N}\right|\right)\leq A\max\left(\left|z_{1}\right|,\ldots,\left|z_{l_-}\right|\right)\right\}.
    \end{align*}
    Then, for any sufficiently large $A>0$, we have the following statement: for any compact set $K\subset \bA^{N,\an}\setminus K_f^+$ and for sufficiently large $n$ (depending on $K$), $f^n(K)\subset W_A$.
\end{prop}

We note that the assumption $\{z_1=\cdots=z_{l_+}\}\cap I(F^{-1})=\emptyset$ can always be assumed to be true by taking a coordinate change. If we set $\tilde{W}_A\subset\bP^{N,\an}_K$ to be the closure of $W_A$, it is a neighborhood of $I(F^{-1})$.

To show the Proposition \ref{prop: regularpolyauto - nonarchdynamics - LimitOfUnboundedPoints}, we need the following lemma:
\begin{lmm}
    With the same notation as Proposition \ref{prop: regularpolyautom - NAdynamics - Kawaguchi's results}, and for any $\delta$ and $\varepsilon$ satisfying the conditions required there, there exists a sequence of positive numbers $\varepsilon_n>0$ such that $\varepsilon_n\to 0$ as $n\to\infty$ and
    \begin{align*}
        f^n(U^+_{\delta,\varepsilon})\subset U_{\delta,\varepsilon_n}^+.
    \end{align*}
\end{lmm}

\begin{proof}
    We already know $f^n(U_{\delta,\varepsilon}^+)\subset U_{\delta,\varepsilon}^+$ and $U_{\delta,\varepsilon'}\subset U_{\delta,\varepsilon}$ for any $\varepsilon'<\varepsilon$, it is enough to calculate how small $\varepsilon'$ can be for each $n$.

    Take any $z\in U_{\delta,\varepsilon}^+$, in particular, $\|z\|>\varepsilon^{-1}$. Then,
    \begin{align*}
        \left\|f(z)\right\|>\delta\|z\|^{d_+}>\frac{\delta}{\varepsilon^{d_+}}=\left(\frac{\delta}{\varepsilon^{d_+-1}}\right)\frac{1}{\varepsilon}.
    \end{align*}
    If we set $\theta_0:=\varepsilon^{d_+-1}/\delta$ and $\varepsilon_1:=\theta_0\varepsilon$, then $\theta_0<1$ and $f(U_{\delta,\varepsilon}^+)\subset U_{\delta,\varepsilon_1}^+$. In the same way, if we set
    \begin{align*}
        &\theta_n:=\frac{\varepsilon_n^{d_+-1}}{\delta};\\
        &\varepsilon_{n+1}=\theta_n\varepsilon_n,
    \end{align*}
    we obtain $f^n(U_{\delta,\varepsilon}^+)\subset U_{\delta,\varepsilon_n}^+$. It is enough to calculate $\theta_n$:
    \begin{align*}
        \theta_n=\frac{\varepsilon_n^{d_+-1}}{\delta}=\frac{(\varepsilon_{n-1}^{d_+}/\delta)^{d_+-1}}{\delta}=\frac{(\theta_{n-1}\varepsilon_{n-1})^{d_+-1}}{\delta}=\theta_{n-1}^n
    \end{align*}
    Since $\theta_0<1$, $\varepsilon_n\to0$.
\end{proof}

\begin{proof}[Proof of Proposition \ref{prop: regularpolyauto - nonarchdynamics - LimitOfUnboundedPoints}]
    Fix any compact set $K\subset \bA^{N,\an}\setminus K_f^+$. Since $\bigcup_n f^{-n}(U_{\delta,\varepsilon}^+)=\bA^{N,\an}\setminus K_f^+$ by Proposition \ref{prop: regularpolyautom - NAdynamics - Kawaguchi's results}, there exists a natural number $N$ such that $f^N(K)\subset U_{\delta,\varepsilon}^+$. Hence, it is enough to show that there exists a natural number $n$ such that $U_{\delta,\varepsilon_n}^+\subset W_A$. Since $\tilde{U}_{\delta,\varepsilon}^+\cap \tilde{N}_{\delta,\varepsilon}^+=\emptyset$, $F$ is well-defined on $\tilde{U}_{\delta,\varepsilon}^*$.

    If we write
    \begin{align*}
        f(x)=p(x)+(l.o.t.), \text{ with }p(x)=(p_1(x),\ldots,p_N(x)),
    \end{align*}
    where $p_i(x)$ is the sum of the monomials of degree $d_+$, $F$ on $(S=0)\setminus I(F)$ is written as
    \begin{align*}
        F(X:0)=[P(X):0], \text{ where }P(X)=S^{d_+}p(X/S).
    \end{align*}
    Since $F((S=0)\setminus I(F))=I(F^{-1})$,
    \begin{align*}
        [P(X):0]\in I(F^{-1})\subset \tilde{W}_A \text{ for any }A>0.
    \end{align*}
    By the continuity of $F$, as $\|z\|\to\infty$, we have the convergence $F(z:1)\to [p(z):0]$. That is, for sufficiently large $R>0$, we obtain $F(z:1)\in \tilde{W}_A$.

    Therefore, if we take $n$ so large that $\varepsilon_n^{-1}>R$, we obtain $f(U_{\delta,\varepsilon_n}^+)\subset W_A$.
\end{proof}

\begin{lmm} \label{lmm: regularpolyauto - nonarchdynamics - GreenFctionIndepOfNorms}
    For any $A>0$ and $m=1,\ldots,N$,
    \begin{align*}
        G_{\NA}^+(z)=\lim_{n\to\infty}\frac{1}{d_+^n}\log^+\max\left(|f_1^n(z)|,\ldots,|f_m^n(z)|,A|f_{m+1}^n(z)|,\ldots,A|f_N^n(z)|\right)
    \end{align*}
\end{lmm}
\begin{proof}
    For any $z\in\bA^{N,\an}_K$,
    \begin{align*}
        \left|\log^+|z|-\log^+\max\left(|z_1|,\ldots,|z_m|,A|z_{m+1}|,\ldots,|z_N|\right)\right|\leq \left|\log A\right|.
    \end{align*}
    Hence,
    \begin{align*}
        \left|\frac{1}{d_+^n}\log^+\|f^n(z)\|-\frac{1}{d_+^n}\log^+\max\left(|f_1^n(z)|,\ldots,|f_m^n(z)|,A|f_{m+1}^n(z)|,\ldots,A|f_N^n(z)|\right)\right|\leq \frac{|\log A|}{d_+^n},
    \end{align*}
    which shows the claim.
\end{proof}

\begin{proof}[Proof of Proposition \ref{prop: regularpolyauto - nonarchdynamics - GreensoutsideofJuliaset}]
    We assume $\{z_1=\cdots=z_{l_-}=0\}\cap I(F^{-1})=\emptyset$, and take $A>0$ so that $\tilde{W}_A\cap I(F)=\emptyset$.

    First, by Proposition \ref{prop: regularpolyauto - nonarchdynamics - LimitOfUnboundedPoints} and Lemma \ref{lmm: regularpolyauto - nonarchdynamics - GreenFctionIndepOfNorms}, on any compact set $K\subset \bA^{N,\an}_K\setminus K_f^+$, we have the following equation:
    \begin{align*}
        G_{\NA}^+(z)=\lim_{n\to\infty}\frac{1}{d_+^n}\log\max\left\{\left|f_1^n(z)\right|,\ldots,\left|f_{l_-}^n(z)\right|\right\},
    \end{align*}
    where the limit is uniform. We note that for each $n$,
    \begin{align*}
        \log\max\left(1,\left|f_1^n(z)\right|,\ldots,\left|f_{l_-}^n(z)\right|\right)=\log \max \left(\frac{1}{\left|f_1^n(z)\right|}, 1, \ldots, \left|\frac{f_{l_-}^n(z)}{f_1^n(z)}\right|\right)+\log\left|f_1^n(z)\right|,
    \end{align*}
    where $f_1$ is non-zero on whole $K$, i.e. invertible. Then, each of $1/f_1^n, f_2^n/f_1^n,\ldots,f_{l_-}^n/f_1^n$ is also holomorphic on $K$, and $\log|f_1^n|$ is pluriharmonic on the interior of $K$. Therefore, Proposition \ref{prop: pluripotential theory - lth wedge of ddclogmax(1,|g1|,...,|gl-1|) is 0} assures
    \begin{align*}
        \left(\dd^c \log\max\left(1,\left|f_1^n(z)\right|,\ldots,\left|f_{l_-}^n(z)\right|\right)\right)^{l_-}=0
    \end{align*}
    for each $n$ on the interior of $K$. Since the limit is uniform, we obtain the claim.

\end{proof}
We note that by switching $f$ by $f^{-1}$, we obtain the same statement for $G_{NA}^-$, that is, on $\bA^{N,\an}_K\setminus K_f^-$,
\begin{align*}
    (\dd^cG_{\NA}^-)^{l_+}=0.
\end{align*}

\subsection{Proof of Theorem \ref{thm: regularpolyauto - nonarchdynamics - invariancemain}}
First we show the invariance under the action of $f$. By \ref{prop: regularpolyautom - fact - l+ + l_- =N, d^l=d^l},
\begin{align*}
    d_+^{l_-}=d_-^{l_+}.
\end{align*}
By the definition of $G_{\NA}^{\pm}$ we see that
\begin{align*}
    (f^{\pm 1})^*(T_{\pm,\NA})=d_{\pm}\cdot T_{\pm,\NA}.
\end{align*}
Combining these identities, we obtain $f^*\mu_{\NA}=\mu_{\NA}$. 

Next we show the equality $\mu_{\NA}=\left(\dd^c\max\{G_{\NA}^{+},G_{\NA}^{-}\}\right)^N$. We put $G_{\NA}=\max(G_{\NA}^+,G_{\NA}^-)$. The equality we need to show is
\begin{align*}
    \left(\dd^c G_{\NA}\right)^N=\left(\dd^c G_{\NA}^+\right)^{l_-}\wedge \left(\dd^c G_{\NA}^-\right)^{l_+}.
\end{align*}
Note that by definition
\begin{align*}
    \left(\dd^c G_{\NA}\right)^N=\dd^c\left(G_{\NA} \cdot \left(\dd^c G_{\NA}\right)^{N-1}\right).
\end{align*}
We first show that
\begin{align*}
    G_{\NA} \cdot \left(\dd^c G_{\NA}\right)^{N-1}=G_{\NA}^+ \cdot \left(\dd^c G_{\NA}\right)^{N-1}.
\end{align*}
This equality is trivial on the set of $z\in\bA^{N,\an}_K$ such that $G_{\NA}(z)=G^+_{\NA}(z)$. When $G_{\NA}(z)>G_{\NA}^+(z)$, by definition we obtain $G^-_{\NA}(z)>G_{\NA}^+(z)\geq 0$. By Proposition \ref{prop: regularpolyauto - nonarchdynamics - GreensoutsideofJuliaset},
\begin{align*}
    \left(\dd^c G_{\NA}\right)^{N-1}=\left(\dd^c G_{\NA}\right)^{l_+}=0 \text{ on }G_{\NA}^-(z)>0,
\end{align*}
from which we see that the both sides equal to $0$. Hence, on the whole $\bA^{N,\an}_K$ we obtain the above equality, and taking $\dd^c$ to obtain
\begin{align*}
    \left(\dd^c G_{\NA}\right)^N=\dd^c G_{\NA}^+ \wedge \left(\dd^c G_{\NA}\right)^{N-1}.
\end{align*}
Since the above argument can be continued as long as $N-i\geq l_+$, we obtain
\begin{align*}
    \left(\dd^c G_{\NA}\right)^N= \left(\dd^c G_{\NA}^+\right)^{l_-} \wedge \left(\dd^c G_{\NA}\right)^{l_+}.
\end{align*}
In the same way as above, we write
\begin{align*}
    \left(\dd^c G_{\NA}^+\right)^{l_-} \wedge \left(\dd^c G_{\NA}\right)^{l_+}=  \dd^c\left( G_{\NA} \left(\dd^c G_{\NA}^+\right)^{l_-} \wedge \left(\dd^c G_{\NA}\right)^{l_+-1}\right),
\end{align*}
and enough to show
\begin{align*}
    G_{\NA}\left(\dd^c G_{\NA}^+\right)^{l_-} \wedge \left(\dd^c G_{\NA}\right)^{l_+-1}=G^-_{\NA}\left(\dd^c G_{\NA}^+\right)^{l_-} \wedge \left(\dd^c G_{\NA}\right)^{l_+-1}.
\end{align*}
Again, the equation is trivial when $G_{\NA}=G_{\NA}^-.$ On the locus $G_{\NA}=G_{\NA}^+>G_{\NA}^-\geq0$, Proposition \ref{prop: regularpolyauto - nonarchdynamics - GreensoutsideofJuliaset} assures
\begin{align*}
    \left(\dd^c G_{\NA}\right)^{l_-}=0 \text{ on }G_{\NA}^+(z)>0,
\end{align*}
whence we obtain the equality. As before, repeating this procedure to show
\begin{align*}
    \left(\dd^c G_{\NA}\right)^N &= \left(\dd^c G_{\NA}^+\right)^{l_-} \wedge \left(\dd^c G_{\NA}\right)^{l_+} \\ 
    &= \left(\dd^c G_{\NA}^+\right)^{l_-} \wedge \left(\dd^c G_{\NA}^-\right)^{l_+},
\end{align*}
which completes the proof.

\section{Uniformity of the complex Green function}\label{sec: regularpolyautom - uniformity of cpx dynamics}
Let $f:\bC^N\times\bD^*\to \bC^N\times\bD^*$ be an analytic family of regular polynomial automorphisms  of degree $d_+$ (and the degree of $f^{-1}$ will be denoted by $d_-$) possibly meromorphically degenerating at the origin. We take the coordinate $(x_1,\ldots,x_N)$ of $\bC^N$ and $(x_1:\cdots:x_N:s)$ of $\bP^N(\bC)$ with the line of infinity defined by $s=0$, and put $F_t^{n}=(F_{1,t}^n:\cdots:F_{N,t}^n:s^{d_+^n}):\bP^N(\bC)\times\bD^*\dashrightarrow\bP^N(\bC)\times\bD^*$ to be the extension of $f^n_t$.

We introduce several notation that is useful in this section. We set $d:=d_+^{l_-}=d_-^{l_+}$ and write
\begin{align*}
    F_{i,t}^{l_-}(x_1,\ldots,x_N,0)=\sum_{|J|=d}a_{i,J}(t)x^J; \\
    F_{i,t}^{-l_+}(x_1,\ldots,x_N,0)=\sum_{|J|=d}b_{i,J}(t)x^J,       
\end{align*}
where $a_{i,J}$ and $b_{i,J}$ are meromorphic function around $0$. We define
\begin{align*}
    M_t=\max_{i,J}\{|a_{i,J}(t)|, |b_{i,J}(t)|,1\}.
    \end{align*}
Since all the coefficients are meromorphic, there exists a constant $A>0$ independent of $t$ such that
\begin{align*}
    (0\leq)\log M_t<A\log|t|^{-1}
\end{align*}
for $|t|<<1$. For the later argument, we take $A$ a bit larger to assume
\begin{align} \label{eq: regularpolyauto - uniformity - eval of M_t}
    \log M_t +C_0 <A\log|t|^{-1},
\end{align}
where $C_0$ is a positive constant independent of $t$, which is determined later.

The following lemma is important to evaluate the uniformity on the line at infinity.
\begin{lmm}\label{prop: regularpolyauto - unifomity - eval at infinity line}
     In the above notation, there exists a constant $C>0$ independent of $t$ such that for any $x=(x_1:\cdots:x_N:0)\in I(f_t)$, we have
    \begin{align*}
        \left|\frac{1}{d}\log|F_t^{-l_+}(x)\|-\log\|x\|\right|\leq C\log|t|^{-1},
    \end{align*}
    where $\|x\|=\max\{|x_1|,\ldots,|x_N|\}$. Similarly, for any $x=(x_1:\cdots:x_N:0)\in I(f_t^{-1})$, we have 
    \begin{align*}
        \left|\frac{1}{d}\log\|F_t^{l_-}(x)\|-\log\|x\|\right|\leq C\log|t|^{-1},
    \end{align*}

    \end{lmm}
Note that $\log\max\left\{\|F_t^{-l_+}(x)\|,\|F_t^{l_-}(x)\|\right\}$ and $\log\|x\|$ are not well-defined though the difference is independent of the choice of the homogeneous coordinate of $x$.

The proof is based on \cite{HindrySilverman2000}, but to show the uniformity on $t$, we need more concrete argument about the constants appearing in their proof.

\begin{proof}

    First we show that for any $(z:0)\in(s=0)\subset\bP^N(\bC)$,
    \begin{align*}
        \frac{1}{d}\log\|F_t^{-l_+}(z,0)\|-\log\|z,0\|\leq \frac{A}{d}\log|t|^{-1}.
    \end{align*}
    This is straightforward evaluation of coefficients:
    \begin{align*}
        \|F_{i,t}^{-l_+}(z,0)\|&=\left\|\sum_{|J|=d}b_{i,J}(t)z^J\right\| \\
        &\leq N_0\max_{J}|b_{i,J}(t)|\max_{i}|z_i|^d \\
        &\leq N_0M_t\|z\|^d,
    \end{align*}
    where $N_0$ is the number of terms of $F_{i,t}^{-l_+}$ as polynomial of $z$. By taking the maximum through $i$ and their logarithm we obtain
    \begin{align}\label{eq: regularpolyauto - uniformity - 1st ineq when s=0}
        \frac{1}{d}\log\|F_t^{-l_+}(z,0)\|-\log\|z\|\leq \log M_t +\log N_0 <A\log|t|^{-1},
    \end{align}
    where the inequality follows from equation (\ref{eq: regularpolyauto - uniformity - eval of M_t}) when $C_0=\log N_0$. The same argument for $f^{l_-}$ shows that
    \begin{align*}
        \frac{1}{d}\log\|F_t^{l_-}(z,0)\|-\log\|z\|<A\log|t|^{-1}.
    \end{align*}
    
    Before showing the opposite inequality, we prepare several notation. Take any $z=(z_1:\cdots:z_N)$. Since $\{F_{i,t}^{-l_+}(z,0),F_{i,t}^{l_-}(z,0)\}_i$ has no common zero, there exists polynomials $p_{ij,t},q_{ij,t}$ and integer $m>d$ such that
    \begin{align}\label{eq: regularpolyautom - uniformity - resultant}
        \sum_{i=1}^N p_{ij,t}(z,0)F_{i,t}^{-l_+}(z,0)+\sum_{i=1}^N q_{ij,t}(z,0)F_{i,t}^{l_-}(z,0)=z_j^m
    \end{align}
    for $j=1,\ldots,N$. Note that $p_{ij,t}, q_{ij,t}$ are homogeneous of degree $m-d$ and their coefficients are written as linear sum of those of $F_{i,t}^{-l_+}$ and $F_{i,t}^{l_-}$, we may regard the coefficients are complex Laurent series around $0$. There exists a positive constant $M_t'>0$ depending on $t$ such that
    \begin{align*}
        \left|\left(\text{coefficients of }p_{ij,t}, q_{ij,t}\right)\right|<M_t',
    \end{align*}
    and we may evaluate $M_t'$ as
    \begin{align}\label{eq: regularpolyauto - uniformity - eval at s=0}
        \log M_t'+ C\leq A' \log |t|^{-1}
    \end{align}
    for any small $t$, where $C$ is some constant defined later independently of $t$. Again we take $A'$ a bit large so that
    \begin{align*}
        \log M_t +C_1<A'\log|t|^{-1}.
    \end{align*}

    Now we take $x\in I(f_t)$. Note that by definition the coordinate of $x$ is of the form $x=(x_1:\cdots:x_N:0)$. Since $x\in I(f_t)=I(f_t^{l_-})$, we obtain
    \begin{align*}
        s=F_{1,t}^{l_-}(x)=\cdots=F_{N,t}^{l_-}(x)=0,
    \end{align*}
    and by the regularity $x\not\in I(f^{-1})$, that is,
    \begin{align*}
        (F_{1,t}^{-l_+}(x),\ldots,F_{N,t}^{-l_+}(x))\neq(0,\ldots,0).
    \end{align*}
    By equation (\ref{eq: regularpolyautom - uniformity - resultant}) when $(z:0)=x$,
    \begin{align*}
        \sum_{i=1}^N p_{ij,t}(x)F_{i,t}^{-l_+}(x)=x_j^m
    \end{align*}
    for $j=1,\ldots,N$. Then,
    \begin{align*}
        |x_j|^m&=\left|\sum_{i=1}^N p_{ij,t}(x)F_{i,t}^{-l_+}(x)\right| \\
        &\leq N_1 \max_j|p_{ij,t}(x)| \max\left|F_{i,t}^{-l_+}(x)\right| \\
        &\leq N_1N_2 M_t' \|x\|^{m-d}\left\|F_{t}^{-l_+}(x)\right\|,
    \end{align*}
    where $N_1,N_2$ are number of terms as $N_0$. By taking the maximum we obtain
    \begin{align*}
        \|x\|^m\leq N_1N_2 M_t' \|x\|^{m-d}\left\|F_{t}^{-l_+}(x)\right\|.
    \end{align*}
    By taking $C=\log N_1N_2$ in (\ref{eq: regularpolyauto - uniformity - eval at s=0}), we obtain
    \begin{align}\label{eq: regularpolyauto - uniformity - 2nd ineq when s=0}
        \log\|x\|-\frac{1}{d}\log\left\|F_{t}^{-l_+}(x)\right\|\leq\frac{1}{d}\log(N_1N_2M_t')\leq \frac{A'}{d}\log|t|^{-1}.
    \end{align}
    Now we define $C=\max\{A/d,A'/d\}$ in (\ref{eq: regularpolyauto - uniformity - 1st ineq when s=0}) and (\ref{eq: regularpolyauto - uniformity - 2nd ineq when s=0}), we obtain the result.

    When $x\in I(f_t^{-1})$, the same proof works if we switch the roles of $p_{ij,t}$ and $q_{ij,t}$.
\end{proof}

We define the function $g:\bP^N(\bC)\times\bD^*\to\bR$ as follows:
\begin{align*}
    g(x,t)=\lim_{n\to\infty}\frac{1}{d^n}\log\max\left\{\left\|F_t^{nl_-}(x)\right\|_*, \left\|F_t^{-nl_+}(x)\right\|_*\right\}.
\end{align*}
The existence of the function $g(x,t)$ is not trivial, but as soon as we show it, clearly it is a limit of a sequence of model functions
\begin{align*}
    g_n(x,t)=\log\max\left\{\left\|F_t^{nl_-}(x)\right\|_*, \left\|F_t^{-nl_+}(x)\right\|_*\right\}.
\end{align*}
In this section, we show not only the pointwise convergence of $g_n$, but also the convergence uniformly in $t$;
\begin{prop} \label{prop: regularpolyautom - uniformity main}
    In the above notation, the limit in the definition of $g$ exists for all $(x,t)\in\bP^N(\bC)\times\bD^*$. Moreover, there exists a positive constant $C>0$ such that
    \begin{align*}
        \sup_{x\in\bP^N(\bC)}\left|g(x,t)-\frac{1}{d^n}g_n(x,t)\right|\leq \frac{C}{d^n}\log|t|^{-1}
    \end{align*}
    for all $t$ sufficiently close to $0$.
\end{prop}
Note that Lee \cite{Lee:2012mb} showed a similar result in Theorem 6.5. His result treats both of the archimedean and non-archimedean dynamics with broader polynomial morphisms, but the uniformity in his setting is slightly different from ours since we need to treat the dependence of $t$. The strategy is similar to his, but to clarify this difference we display the proof.
\begin{proof}
    We put for any $(x,t)\in\bP^N(\bC)\times\bD^*$ and integers $n>m>0$,
    \begin{align*}
        \Delta_{m,n}(x,t):=\left|\frac{1}{d^n}g_n(x,t)-\frac{1}{d^m}g_m(x,t)\right|.
    \end{align*}
    Note that when $x\in\bP^N(\bC)$ lies in the line of infinity, the proposition follows from Lemma \ref{prop: regularpolyauto - unifomity - eval at infinity line}. Indeed, for any two integers $n>m$ and for any $(x,t)=([x_1:\ldots:x_N:0],t)$, we set
    \begin{align*}
        \Delta_{m,n}(x,t)&=\left|\frac{1}{d^n}g_n(x,t)-\frac{1}{d^m}g_m(x,t)\right| \\
        &= \frac{1}{d^m}\left|\frac{1}{d^{n-m}}\log\max\left\{\|F_t^{-nl_+}(x)\|,\|F_t^{nl_-}(x)\|\right\}-\log\max\left\{\|F_t^{-ml_+}(x)\|,\|F_t^{ml_-}(x)\|\right\}\right|.
    \end{align*}
    When $x\in I(f_t)\cup I(f_t^{-1})$, we may apply Lemma \ref{prop: regularpolyauto - unifomity - eval at infinity line} directly to obtain
    \begin{align*}
        \Delta_{m,n}(x,t)\leq\left(\sum_{i=m}^{n-1}\frac{C}{d^i}\right)\log|t|^{-1},
    \end{align*}
    which means the sequence $\{g_n(x,t)/d^n\}$ is Cauchy, and the limit satisfies
    \begin{align*}
        \sup_{x\in I(f_t)\cup I(f_t^{-1})}\left|g(x,t)-\frac{1}{d^n}g_n(x,t)\right|\leq \left(\sum_{i\geq n}\frac{C}{d^i}\right)\log|t|^{-1}=\frac{C}{d-1}\cdot\frac{1}{d^n}\log|t|^{-1}.
    \end{align*}
    When $x\in (s=0)\setminus\bigl(I(f_t)\cup I(f_t^{-1})\bigr)$, we note that $f_t(x)\in I(f_t^{-1})$ and $f_t^{-1}(x)\in I(f_t)$. Therefore for any $n>0$, Lemma \ref{prop: regularpolyauto - unifomity - eval at infinity line} is applicable to obtain
    \begin{align*}
        \left|\frac{1}{d^{n+1}}\log\left\|F_t^{-(n+1)l_+}(x)\right\|-\frac{1}{d^n}\log\left\|F_t^{-nl_+}(x)\right\|\right|\leq \frac{C}{d^n}\log|t|^{-1}; \\
        \left|\frac{1}{d^{n+1}}\log\left\|F_t^{(n+1)l_-}(x)\right\|-\frac{1}{d^n}\log\left\|F_t^{nl_-}(x)\right\|\right|\leq \frac{C}{d^n}\log|t|^{-1};
    \end{align*}
    Therefore by taking max, we obtain
    \begin{align*}
        \left|\frac{1}{d^{n+1}}\log\max\left\{\left\|F_t^{-(n+1)l_+}(x)\right\|,\left\|F_t^{(n+1)l_-}(x)\right\|\right\}-\frac{1}{d^{n}}\log\max\left\{\left\|F_t^{-nl_+}(x)\right\|,\left\|F_t^{nl_-}(x)\right\|\right\} \right|\leq \frac{C}{d^n}\log|t|^{-1}.
    \end{align*}
    This leads us to show that for any $n>m>0$,
    \begin{align*}
        \Delta_{m,n}(x,t)\leq \left(\sum_{i=m}^{n-1} \frac{C}{d^i}\right)\log|t|^{-1},
    \end{align*}
    and hence
    \begin{align}\label{eq: regularpolyautom - uniformity - eval at infty}
        \sup_{x\in (s=0)}\left|g(x,t)-\frac{1}{d^n}g_n(x,t)\right|\leq \left(\sum_{i\geq n}\frac{C}{d^i}\right)\log|t|^{-1}=\frac{C}{d-1}\cdot\frac{1}{d^n}\log|t|^{-1}.
    \end{align}
    Next we evaluate $\Delta_{m,n}(x,t)$ on $\bC^N$. To show this, we prepare several notations. First we define the following two functions on $\bC^{2N}$: for $(x,y)\in \bC^N\times\bC^N$,
    \begin{align*}
        \Phi_t(x,y):=(f_t^{l_-},f_t^{-l_+})(x,y)=(f_t^{l_-}(x),f_t^{-l_+}(y)); \\
        \Psi_t(x,y):=(f_t^{-l_+},f_t^{l_-})(x,y)=(f_t^{-l_+}(x),f_t^{l_-}(y)).
    \end{align*}
    
    We note that the sets of coefficients of $\Phi$ and $\Psi$ are exactly the union of those of $f_t^{-l_+}$ and $f_t^{l_-}$. Let $C'(t)$ be the maximum of the absolute values of all the coefficients of $f_t^{-l_+}$ and $f_t^{l_-}$, and $C(t):=\max(C'(t),1)$. Take
    \begin{align*}
        &\delta(t):=\frac{1}{C(t)^{d(d-1)}}(<1); \\
        &\varepsilon(t):=\frac{1}{C(t)^d}(<1).
    \end{align*}
    Then, since each coefficient of $f_t^{-l_+}$ or $f_t^{l_-}$ is meromorphic on $\bD$ with possible pole only at the origin, there exists a positive number $A',B'>0$ independent of $t$ such that
    \begin{align*}
        -\log\delta(t)&\leq A'\log|t|^{-1}; \\
        \log C(t)&\leq B'\log|t|^{-1}
    \end{align*}
    around $0$. For a map $\Phi_t$, we define $V_{\Phi_t}$ similarly as before:
    \begin{align*}
        V_{\Phi_t}=\left\{z\in \bC^{2N}\middle| \|z\|\leq\varepsilon(t)^{-1}\text{ or }\|\Phi_t(z)\|\geq \delta(t)\|z\|^d\right\}.
    \end{align*}
    Then, for any $(x,y)\in V_{\Phi_t}$, it follows that
    \begin{align}\label{eq: regularpolyauto - uniformity - unbounded region}
        \log^+\max\left(|x|,|y|\right)-\frac{1}{d}\log^+\max\left(\left|f_t^{l_-}(x)\right|,\left|f_t^{-l_+}(y)\right|\right)\leq -\log \delta(t)\leq A'\log|t|^{-1}.
    \end{align}
    In the same way, if we define $V_{\Psi_t}$ similarly, for any $(x,y)\in V_{\Psi_t}$, it follows that
    \begin{align*}
        \log^+\max\left(|x|,|y|\right)-\frac{1}{d}\log^+\max\left(\left|f_t^{-l_+}(x)\right|,\left|f_t^{l_-}(y)\right|\right)\leq -\log \delta(t)\leq A'\log|t|^{-1}.
    \end{align*}

    Now we fix $x=(x_1,\ldots,x_N)\in\bC^N$. By definition
    \begin{align*}
        &\left|f_t^{-l_+}(x)\right|\leq C'(t)\cdot \|x\|^d; \\
        &\left|f_t^{l_-}(x)\right|\leq C'(t)\cdot \|x\|^d.
    \end{align*}
    Taking the logarithm we obtain
    \begin{align*}
        \frac{1}{d}\log^+\max\left(\left|f_t^{-l_+}(x)\right|,\left|f_t^{l_-}(x)\right|\right) -\log^+|x|\leq \log C(t)\leq B'\log|t|^{-1}.
    \end{align*}

    We show that $x\in V_{\Phi_t}\cap V_{\Psi_t}$. We note that by definition,
    \begin{align*}
        V_{\Phi_t}\cap \Delta =V_{\Psi_t}\cap \Delta,
    \end{align*}
    where
    \begin{align*}
        \Delta=\{(x,x)\in\bC^N\times\bC^N\},
    \end{align*}
    since the defining inequalities coincide in this case.
    Set
    \begin{align*}
        V_t^{\pm}=\left\{x\in\bC^N\middle| \|x\|\leq \varepsilon(t)^{-1}\text{ or } \|f_t^{l_{\mp}}(x)\|\geq \delta(t)\|z|^d\right\}.
    \end{align*}
    
    Since $V^+_t\cup V^-_t=\bC^N$, we must have either $\max(|f^{-l_+}_t(x)|,1)\geq\delta(t)\max(|x|^d,1)$ or $\max(|f^{l_-}_t(x)|,1)\geq\delta(t)\max(|x|^d,1)$, that is, $(x,x)\in V_{\Phi_t}$.

    We apply (\ref{eq: regularpolyauto - uniformity - unbounded region}) to $(x,x)$,
    \begin{align*}
        \log^+|x|-\frac{1}{d}\log^+\max\left(\left|f_t^{-l_+}(x)\right|,\left|f_t^{l_-}(x)\right|\right)\leq A'\log|t|^{-1}.
    \end{align*}
    Therefore we obtain for any $x\in\bC^N$,
    \begin{align*}
        \left|\log^+|x|-\frac{1}{d}\log^+\max\left(\left|f_t^{-l_+}(x)\right|, \left|f_t^{l_-}(x)\right|\right)\right|\leq \max(A',B')\log|t|^{-1}.
    \end{align*}
    Now we evaluate $\Delta_{m,n}(x,t)$. The point $x\in\bC^N$ is written as $(x_1:\cdots:x_N:1)$ in the homogeneous coordinate of $\bP^N(\bC)$. Hence 
    \begin{align*}
        \log \|F_t^{-l_+}(x_1:\cdots:x_N:1)\|=\log^+|f_t^{-l_+}(x_1,\ldots,x_N)|,
    \end{align*}
    which shows
    \begin{align*}
        \Delta_{m,n}(x,t)\leq \left(\sum_{i=m}^{n-1}\frac{\max(A',B')}{d^{i}}\right)\log|t|^{-1},
    \end{align*}
    and hence
    \begin{align}\label{eq: regularpolyauto - uniformity - eval at affine}
        \left|g(x,t)-\frac{1}{d^n}g_n(x,t)\right|\leq \frac{\max(A',B')}{d-1}\frac{1}{d^n}\log|t|^{-1}
    \end{align}
    By (\ref{eq: regularpolyautom - uniformity - eval at infty}) and (\ref{eq: regularpolyauto - uniformity - eval at affine}), we obtain
    \begin{align*}
        \sup_{x\in \bP^N(\bC)}\left|g(x,t)-\frac{1}{d^n}g_n(x,t)\right|\leq\frac{C}{d^n}\log|t|^{-1}
    \end{align*}
    by replacing the constant.
\end{proof}

\section{Hybrid convergence of the invariant measures}\label{sec: regularpolyautom - Hybrid convergence of the invariant measures}
In this section, we state and show our main theorem. First we prepare the notation.

Let $f:\bC^N\times\bD^*\to \bC^N\times\bD^*$ be an analytic family of regular polynomial automorphisms possibly meromorphically degenerating at the origin and $f_{\NA}:=f_{\bC((t))}^{\an}:\bA^{N,\an}_{\bC((t))}\to\bA^{N,\an}_{\bC((t))}$ be the induced non-archimedean automorphism. Fix a small positive number $r>0$. As in Section \ref{sec: intro of berkoich spaces}, we can consider the hybrid analytification of the affine space $\bA^{N,\hyb}$. We have a proper surjection $\pi_{\hyb}:\bA^{N,\hyb}\to\cC_{\hyb}(r)$, a homeomorphism $\tau:\bar{\bD}_r\to\cC_{\hyb}(r)$, and a family of homeomorphisms $\{\psi_t\}_{t\in\bar{\bD}_r}$ such that
\begin{align*}
    &\psi_t: \bC^N\to \pi_{\hyb}^{-1}\left(\tau(t)\right) \quad \text{for}\quad t\neq0;\\
    &\psi_0: \bA^{N,\an}_{\bC((t))}\to \pi_{\hyb}^{-1}\left(\tau(0)\right).
\end{align*}
For each $t$, we have the invariant measure $\mu_t=T_{t,+}^{l_-}\wedge T_{t,-}^{l_+}$ by \cite{Sibony99}. For the induced non-archimedean automorphism $f_{\NA}$, we constructed an analogous invariant measure $\mu_{\NA}$ on $\bA^{N,\an}_{\bC((t))}$ in Theorem \ref{thm: regularpolyauto - nonarchdynamics - invariancemain}. We obtain a family of measures $\{(\psi_t)_*\mu_t\}_{t\in\bar{\bD}^*_r}\cup \{(\psi_0)_*\mu_{\NA}\}$ by pushing them forward to the hybrid affine space. Then, our main result is as follows:
\begin{thm}\label{thm: regularpolyautom - hybrid conv - main theorem}
    In the above notation, we have the following weak convergence of measures:
    \begin{align*}
        \lim_{t\to0}(\psi_t)_*\mu_t=(\psi_0)_*\mu_{\NA}.
    \end{align*}
\end{thm}

The sketch of proof is as follows. We first note that, for each $t\in\bD^*$, the invariant measure $\mu_t$ is written as
\begin{align*}
    \mu_t=(\dd^c G_t)^N
\end{align*}
on $\bC^N$ by Proposition 2.9 of \cite{Guedj:2002yq}, we note that regular polynomial automorphisms in our situation are ``0-regular biholomorphism" in their paper. In the non-archimedean situation, we show the same equality in Theorem \ref{thm: regularpolyauto - nonarchdynamics - invariancemain}. Namely, 
\begin{align*}
    \mu_{\NA}=\left(\dd^c \max\{G_{\NA}^+,G_{\NA}^-\}\right)^N.
\end{align*}
Using this equality, we will show the theorem in the following way: we consider measures $\mu_t$ and $\mu_{\NA}$ on those over the projective spaces, and write them as Monge-Amp\`ere measures of some semi-positive metric of line bundles. By showing the uniformity, we may apply Boucksom-Favre-Jonsson's method to our situation. In this argument, the uniformity part is already shown in Section \ref{sec: regularpolyautom - uniformity of cpx dynamics}.

\subsection{Extension to the projective space}

Let \(f\), \(f_{\NA}\), \(\mu_t\), and \(\mu_{\NA}\) be as above. We use the same notation $\psi_t: \bP^N(\bC)\to \bP^{N,\hyb}$ and $\psi_0: \bP^{N,\an}_{\bC((t))}\to\bP^{N,\hyb}$ that is a natural extension of the original $\psi_t$ and $\psi_0$. In this subsection, we reduce the convergence problem on the hybrid affine
space to the corresponding convergence problem on the hybrid projective space.

We consider extension of $f_t$ and $f_{\NA}$ to the projective space. For the fixed inhomogeneous coordinate $x=(x_1,\ldots,x_N)$ of $\bA^{N}$, we fix a homogeneous coordinate $X=(X_1:\cdots:X_N:S)$ of $\bP^N$ such that $x_i=X_i/S$, and the line of infinity is defined by $S=0$.

As before, set $d_{\pm}:=\deg f^{\pm1}$. By Prop \ref{prop: regularpolyautom - fact - l+ + l_- =N, d^l=d^l}, we set $d:=d_+^{l_-}=d_-^{l_+}.$ Put
\begin{align*}
    &G_{n,t}^{\pm}(x)=\frac{1}{d^n}\log^+\left\|f_t^{\pm nl_{\mp}}(x)\right\|; \\
    &G_{n,t}:=\max G_{n,t}^{\pm}; \\
    &G_t:=\max G_t^{\pm}.
\end{align*}
and the same way for $G_{n,\NA}^{\pm}$, $G_{n,\NA}$, and $G_{\NA}$. Then, we already have
\begin{align*}
    &\mu_t=\left(\dd^c G_t\right)^N; \\
    &\mu_{\NA}=\left(\dd^c G_{\NA}\right)^N.
\end{align*}
Now we extend these functions and measures to the projective space. Define
\begin{align*}
    \phi_n(X,t)=\log\max\left\{\left\|F_t^{nl_-}(X)\right\|_*,\left\|F_t^{-nl_+}(X)\right\|_*\right\}
\end{align*}
where $\|\cdot\|_*$ is a normalized Fubini-Study norm on $\bP^N(\bC)$. By Proposition \ref{prop: regularpolyautom - uniformity main}, we obtain
\begin{align*}
    \frac{1}{d^n}\phi_n\to \phi
\end{align*}
for some continuous metric $\phi$ uniformly in the sense of Favre, Proposition \ref{prop:themainfact}.

Note that $\phi_n$ is a model function attached to the following $\cF_n$ for each $n$:
\begin{itemize}
    \item the degree $d^n$;
    \item meromorphic sections $F_{1,t}^{nl_-},\ldots, F_{N,t}^{nl_-}, F_{1,t}^{-nl_+},\ldots,F_{N,t}^{-nl_+}, S^{d^n}$
    \item the fractional ideal $\fa_n$ generated by the above sections
    \item any log resolution of $\fa_n$.
\end{itemize}
The non-archimedean model function $g_n$ of $\cF_n$ is defined as in \ref{subsection : backgrounds - model functions}. On $(S\neq0)$, the fractional $\fa_n$ is generated by the components of $f^{nl_-}$ and $f^{-nl_+}$; that is,
\begin{align*}
    g_n(x)=G_{n,\NA}(x).
\end{align*}
By Proposition \ref{prop:themainfact}, $g_n$ also has a limit, denoted by $g$.

Observe that both $\phi$ and $g$ have continuous potentials, they do not have a mass on any codimension $1$ subvariety on $\bP^N$.

We define
\begin{align*}
    &\hat{\mu}_t:=\MA_t(g);\\
    &\hat{\mu}_{\NA}:=\MA_{\NA}(g).
\end{align*}
Then we see that
\begin{align*}
    &\hat{\mu}_t|_{\bC^N}=\mu_t; \\
    &\hat{\mu}_{\NA}|_{\bA^{N,\an}_{\bC((t))}}=\mu_{\NA}.
\end{align*}

We recall the lower estimate of the Green function that will be used to control the supports of the measures uniformly. Note that the below statement is shown by Sibony for the complex dynamics and Kawaguchi for the non-archimedean dynamics.

\begin{prop}[\cite{Shu:2013ma},  Theorem 2.3 and \cite{Sibony99}, 2.2.]\label{prop:regularpolyautom-Gf-growth}
Let $K$ be either $\bC$ or a complete non-archimedean field, $f:\bA^N_K\to\bA^N_K$ be a regular polynomial automorphism. For $f$, there exist two open subsets $V^{\pm}$ of $\bA^{N,\an}_K$ such that they cover the whole $\bA^{N,\an}_K$ and  
\begin{align*}
    &G_f^+(x)\geq \log^+\|x\|+c^+\text{ on }V^+; \\
    &G_f^-(x)\geq \log^+\|x\|+c^-\text{ on }V^-.
\end{align*}
\end{prop}

We note that, in the proof of Proposition \ref{prop:regularpolyautom-Gf-growth} of \cite{Shu:2013ma}, $c^{\pm}$ are determined explicitly by the logarithm of the coefficients of $f^{\pm 1}$. In particular, if $f=f_t$ where $f_t$ is an analytic family, there exists a constant $A>0$ such that
\begin{align*}
    \left|c^{\pm}\right|\leq A\log|t|^{-1}.
\end{align*}

Using this fact, we show the following lemma:

\begin{lmm}\label{lmm: regularpolyautom - hybridconv - extn to proj}
There exists a compact subset $W\subset \mathbb A^{N,\hyb}$ such that $\Supp((\psi_t)_*\mu_t)\subset W$ for all sufficiently small $t\neq0$ and $\Supp((\psi_0)_*\mu_{\NA})\subset W$.
\end{lmm}

\begin{proof}
We note that, since $\mu_t=(\dd^c G_t)^N$, $\Supp \mu_t\subset (G_t=0)$.

Fix any $t$ and let $V_t^{\pm}$ be the open set given in Proposition \ref{prop:regularpolyautom-Gf-growth}. For any $x\in V_t^+$, we have
\begin{align*}
    G_t^+(x)\geq \log^+\|x\|+c^+,
\end{align*}
and $x\in V_t^-$
\begin{align*}
    G_t^-(x)\geq \log^+\|x\|+c^-.
\end{align*}
Since $V_t^+\cup V_t^-=\bC^N$, for any $x\in \bC^N$,
\begin{align*}
    G_t(x)\geq\log\|x\|+\min (c^{\pm}).
\end{align*}
Therefore, if $G_t(x)=0$, we must have
\begin{align*}
    \log\|x\|\leq -\min c^{\pm}\leq A\log|t|^{-1}.
\end{align*}
The same formula holds for the non-archimedean dynamical Green function $G_{\NA}$.
Define the compact subset $W$ of $\bA^{N,\hyb}$ as follows:
\begin{align*}
    W=\{x\in \bA^{N,hyb}| \log|x_i|_{\hyb}\leq A(-\log r)\text{ for }i=1,\ldots,N\}. 
\end{align*}
Then by the definition of the hybrid norm and above calculation, we obtain $\psi_t(G_t=0)\subset W$ and $\psi_0(G_{\NA}=0)\subset W$.
\end{proof}

By the above lemma, we may reduce the affine convergence to the projective convergence. That is, the weak convergence $(\psi_t)_*\hat{\mu}_t\to (\psi_0)_*\hat{\mu}_{\NA}$ actually implies $(\psi_t)_*\mu_t\to (\psi_0)_*\mu_{\NA}$. Indeed, since $\bA^{N,\hyb}$ is locally compact and Hausdorff, we may choose $\chi\in C_c(\bA^{N,\hyb})$ such that $\chi|_W=1$.
For any $\phi \in C_b(\bA^{N,\hyb})$, hence,
\begin{align*}
    \int_{\bA^{N,\hyb}}\phi d(\psi_t)_*\mu_t=\int_{\bA^{N,\hyb}}\chi\phi d(\psi_t)_*\mu_t=\int_{\bP^{N,\hyb}}\phi d(\psi_t)_*\hat{\mu}_t
\end{align*}
and the same for $t=0$ and $\mu_{\NA}$. The convergence of the right hand side assures that of the left.

\subsection{Proof of Theorem \ref{thm: regularpolyautom - hybrid conv - main theorem}}
Now we show the main theorem.
By Lemma \ref{lmm: regularpolyautom - hybridconv - extn to proj}, it is enough to show the weak convergence
\begin{align*}
    \lim_{t\to 0}(\bar\psi_t)_*\widehat{\mu}_t=(\bar\psi_0)_*\widehat{\mu}_{\NA}.
\end{align*}

By definition we have
\begin{align*}
    &\MA_x(g(x,t))=\widehat{\mu}_t \quad \text{on}\quad \bP^N(\bC);\\
    &\MA_{\NA}(g)=\widehat{\mu}_{\NA}\quad \text{on}\quad \bP^{N,\an}_{\bC((t))},
\end{align*}
where by $\MA_x$ we denote the Monge-Amp\`ere operator with respect to the variable $x$. By Proposition \ref{prop: regularpolyautom - uniformity main}, $g(x,t)$ is uniformly approximated by the sequence $\{g_n(x,t)/d^n\}_n$, where $g_n$ is a model function of degree $d^n$. Also by definition $g$ is the limit of the sequence of the non-archimedean model function corresponding to $\{g_n(x,t)\}_n$. Therefore by Proposition \ref{prop:themainfact}, we obtain the required weak convergence.

\bibliographystyle{plain}
\bibliography{hybrid_dynamics_of_regular_polynomial_automorphisms}
\end{document}